\documentclass[11pt]{amsart}

\usepackage[vmargin=3cm, hmargin=3cm]{geometry}
\usepackage{url}

\usepackage{enumerate}
\usepackage{pgfplots}
\usetikzlibrary{patterns}
\usepackage{multicol}
\usepackage{cancel}
\usepackage{color}
\usepackage{todonotes}
\usepackage{hyperref}
\usepackage{amssymb}

\newtheorem{theorem}{Theorem}
\newtheorem{proposition}[theorem]{Proposition}
\newtheorem{lemma}[theorem]{Lemma}
\newtheorem{corollary}[theorem]{Corollary}
\theoremstyle{definition}
\newtheorem{example}{Example}
\theoremstyle{remark}
\newtheorem{remark}[theorem]{Remark}
\newtheorem{question}[theorem]{Question}

\newcommand{\C}{\mathrm{C}}
\newcommand{\K}{\mathrm{K}}

\title{Good subsemigroups of $\mathbb N^n$}
\author[D'Anna]{Marco D'Anna}
\address{Dipartimento di Matematica e Informatica, Viale A. Doria, 6 - I 95125 – Catania, Italia}
\email{mdanna@dmi.unict.it}

\author[Garc\'ia-S\'anchez]{Pedro A. Garc\'ia-S\'anchez}
\address{Departamento de \'Algebra and CITIC-UGR, Universidad de Granada, E-18071 Granada, Espa\~na}
\email{pedro@ugr.es} %\url{www.ugr.es/local/pedro}

\thanks{The second author is supported by the projects MTM2014-55367-P, which is funded by Ministerio de Econom\'{\i}a y Competitividad and Fondo Europeo de Desarrollo Regional FEDER, and by the Junta de Andaluc\'{\i}a Grant Number FQM-343}

\author[Micale]{Vincenzo Micale}
\address{Dipartimento di Matematica e Informatica, Viale A. Doria, 6 - I 95125 – Catania, Italia}
\email{vmicale@dmi.unict.it}

\thanks{The authors would like to thank B. A. Heredia for his comments and suggestions}

\author[Tozzo]{Laura Tozzo}
\address{Dipartimento di Matematica, Universita' degli Studi di Genova, Via Dodecaneso 35 - I 16154 - Genova, 	Italia}
\curraddr{Department of Mathematics, University of Kaiserslautern - D 67663 - Kaiserslautern, Deutschland}
\email{tozzo@mathematik.uni-kl.de}

\keywords{Good semigroup, Arf closure,  semigroup of values, algebroid curve}
\subjclass[2010]{20M14, 13A18, 14H50,20-04}
\begin{document}

\begin{abstract}
Value semigroups of non irreducible singular algebraic curves and their fractional ideals are submonoids of $\mathbb Z^n$ that are closed under infimums, have a conductor and fulfill a special compatibility property on their elements. Monoids of $\mathbb N^n$ fulfilling these three conditions are known in the literature as good semigroups and there are examples of good semigroups that are not realizable as the value semigroup of an algebraic curve. In this paper we consider good semigroups independently from their algebraic counterpart, in a purely combinatoric setting. We define the concept of good system of generators, and we show that minimal good systems of generators are unique. 
Moreover, we give a constructive way to compute the canonical ideal and the Arf closure of a good subsemigroup when $n=2$.
\end{abstract}
\maketitle

\section{Introduction}

The notion of good semigroup has been given formally in \cite{a-u}, where the authors studied the properties of value semigroups of a one dimensional analytically unramified ring, that is, of a singularity of an algebraic curve. The properties of these semigroups were already considered in \cite{two, c-d-gz, C-D-K, danna1, felix, delgado, garcia}, but it was in \cite{a-u} that it  was proved that the class of good semigroups is larger than the one of value semigroups. Hence, such semigroups are relevant by their own and they form a natural generalization of numerical semigroups. However, they are more difficult to study than the numerical semigroups, mainly because they are not finitely generated as monoids, and they are not closed under finite intersections. In spite of this, good semigroups can be described by means of a finite set of their elements. 

For value semigroups of algebroid curves there are several approaches in the literature to describe these monoids by means of a finite set of data. 
In \cite{garcia,waldi}, singularities with only two branches are studied, and the finite set considered is the set of maximal elements. Then the value semigroup coincides with the elements in the Cartesian product of the value semigroups of the branches that do not share a coordinate with a maximal element and have the other bigger than this maximal element. 
This approach has been generalized to the case of more than two branches in \cite{felix}. An alternative can be found in \cite{c-d-gz}, where the authors introduce $w$-generators for planar algebroid curves. In this setting, it is shown that the value semigroup can be described by a finite set of $w$-generators (not necessarily belonging to the semigroup) and a boolean expression. For the non planar case, we refer to \cite{a-u, two,arf,C-D-K, felix}.
In particular, in \cite{arf} the authors show that the equimultiplicity class of a singularity can be determined using a finite set of data. This data is equivalent to give a good semigroup satisfying the Arf property.

Our approach is different in nature and takes advantage of the algebraic structure of good semigroups. Moreover, it does not only apply to value semigroups of singular algebraic curves (with any number of branches and any embedding dimension), but also to good semigroups not realizable as value semigroups of curves. In particular, our approach allows to develop useful computational tools to study good semigroups, which was one of the motivations of this study. 

A first idea is to consider the \emph{small elements} of the semigroup, that is, the set of elements between $0$ and the conductor of the semigroup with the usual partial order. It is easy to see that the set of these elements, denoted by $\mathrm{Small}(S)$, determines the semigroup. Therefore, the next natural step is to study subsets $G \subsetneq\mathrm{Small}(S)$, from which is possible to recover completely the semigroup $S$. We define such a subset $G$ to be a \emph{good generating system}.
We call $G$ \emph{minimal} when none of its proper subsets is a good generating system.
It is natural to ask if these minimal generating systems are unique: we prove that this is true in the local case (Theorem \ref{th:unique}), as happens in the
``classical'' setting of cancellative monoids. 
The same is not true in general for the non local case, but it is possible to reduce to the local case.

The structure of the paper is the following.

In Section \ref{sec:good-sem} we recall the concept of good
semigroups and how to obtain them in different ways (see for instance Example \ref{ex:two-good}). 
Given a good semigroup $S$, we define the set of its small elements $\mathrm{Small}(S)$. 
Since $S$ is fully determined by $\mathrm{Small}(S)$ (Proposition \ref{smallelements}), we deduce a first membership test (Proposition
\ref{prop:membership}).

In Section \ref{sec:good-gen-sys} we define the concept of minimal good generating system for $S$. 
In the local case we prove that minimal good generating systems are unique (Theorem \ref{th:unique}). 

Section \ref{sec:ideals} generalizes the results of the previous
section to good relative ideals of $S$. We define the concept of minimal good generating system for a good relative ideal, and we prove that  minimal good generating systems are unique (Theorem \ref{th:uniqueideal}). 
Then we consider the special case of the canonical ideal of $S$. Canonical ideals are important as they play
a crucial role in many properties of good semigroups
\cite{danna1,K-T-S}. 
We give a constructive way to compute the canonical ideal in the two dimensional case, by finding a (non minimal, in general) good
generating system of generators (Proposition \ref{gens-canonical}). 

Finally, in Section \ref{sec:arf} we consider Arf good semigroups, which are a natural generalization of the concept of Arf numerical semigroup. We give an effective method to verify if a good semigroup has the Arf property (Proposition \ref{effective-Arf}). Then we prove that the Arf good closure of a good semigroup always exists (Corollary \ref{Arf-cl}) and, in the two dimensional case, we show how to calculate it.

The procedures presented here have been implemented in \texttt{GAP} \cite{gap} for good semigroups in $\mathbb N^2$, and will be part of the forthcoming stable release of the package \texttt{numericalsgps} \cite{numericalsgps}. The code is available in the development version of the package: \url{https://bitbucket.org/gap-system/numericalsgps} in the file \texttt{good-semigroups.gi} located in the folder \texttt{gap}. Also the functions related to good semigroups are documented in Chapter 12 of the manual in that version (folder \texttt{doc}).

\section{Good semigroups and their ideals}\label{sec:good-sem}

Let $\mathbb N$ be the set of nonnegative integers. As usual, $\le$ stands for the usual partial ordering in $\mathbb N^n$: $a\le b$ if $a_i\le b_i$ for any $i\in\{1,\dots,n\}$. 
Given $a,b\in \mathbb N^n$, the infimum of the set $\{a,b\}$ (with respect to $\le$) will be denoted by $a\wedge b$. Hence $a\wedge b=(\min(a_1,b_1),\dots,\min(a_n,b_n))$.

Let $S$ be a submonoid of $(\mathbb N^n,+)$. We say that $S$ is a \emph{good semigroup} if

\begin{enumerate}[({G}1)]
\item for all $a,b\in S$, $a\wedge b\in S$,
\item if $a,b\in S$ and $a_i=b_i$ for some $i\in \{1,\dots,n\}$, then there exists $c\in S$ such that $c_i>a_i=b_i$, $c_j\ge\min\{a_j,b_j \}$ for $j\in\{1,\dots,n\}\setminus\{i\}$ and $c_j=\min\{a_j,b_j \}$ if $a_j\ne b_j$,
\item there exists $\C\in S$ such that $\C+\mathbb N^n\subseteq S$.
\end{enumerate}

In light of \cite[Proposition 2.1]{a-u}, value semigroups of analitically unramified residually rational one dimensional semilocal rings are good semigroups.

Condition (G1) is denoted as Property A in \cite{garcia}, while
(G2) corresponds to Property B in that paper.

A \emph{relative ideal} of a good semigroup $S$ is a subset $\emptyset\ne E \subseteq \mathbb Z^n$ such that $E + S \subseteq E$ and $a+E\subseteq S$ for some $a\in S$. 
If $E$ satisfies (G1) and (G2), then we say that $E$ is a \emph{good relative ideal} of $S$ (condition (G3) follows from the definition of good relative ideal).

Notice that from condition (G1), if $\C_1$ and $\C_2$ fulfill (G3), then so does $\C_1\wedge \C_2$. So there exists a minimum $\C\in \mathbb N^n$ for which condition (G3) holds. Therefore we will say that
\[
\C(E):=\min\{a\in \mathbb Z^n\mid a+\mathbb N^n\subseteq E\}
\]
is the \emph{conductor} of $E$.
We denote $\gamma(E):=\C(E)-\textbf{1}$ and we abbreviate $\C:=\C(S)$ and $\gamma:=\gamma(S)$, when there is no possible confusion.

For every good relative ideal $E$ of a good semigroup $S$ we define the set of \emph{small elements} as
\[
\mathrm{Small}(E)=\{a\in E\mid a\le \C(E)\}.
\]
In particular, if $E=S$, we have
\[
\mathrm{Small}(S)=\{a\in S\mid a\le \C\}.
\]
Notice that $\C(E)\in\mathrm{Small}(E)$ for all $E$.

For planar curves with two branches $f=f_1 f_2$, if $v$ is the intersection multiplicity of $f_1$ and $f_2$, then $(v+c_1,v+c_2)$ is the conductor of the value semigroup of the curve, where $c_i$ is the conductor of the value semigroup of $f_i$, $i\in\{1,2\}$ (see for instance \cite[Th\'eor\`eme 1.2.6]{salem-th}).
%\todo[inline]{add reference for this fact?}

There are several ways in the literature to obtain good semigroups, and each of them has its own membership procedures, and methods for computing their conductors. We see some of them in the two branches case in the following example.

\begin{example}\label{ex:two-good}
% that do not need of the calculation of the small elements of the semigroup.
	
\noindent a) If we have a numerical semigroup $S$ and a relative ideal $E$ of $S$ with $E\subseteq S$,
then the \emph{semigroup duplication} $S\bowtie E$ is a good semigroup defined as:
\[
S\bowtie E= D\cup (E\times E)\cup\{ a\wedge b\mid a\in D, b\in E\times E\},
\]
where $D=\{(s,s) \mid s\in S\}$ (see \cite{danna}).
	
\begin{verbatim}
gap> s:=NumericalSemigroup(2,3);;
gap> e:=6+s;;
gap> dup:=SemigroupDuplication(s,e);
<Good semigroup>
gap> SmallElements(dup);
[ [ 0, 0 ], [ 2, 2 ], [ 3, 3 ], [ 4, 4 ], [ 5, 5 ], [ 6, 6 ],
[ 6, 7 ], [ 6, 8 ], [ 7, 6 ], [ 7, 7 ], [ 8, 6 ], [ 8, 8 ] ]
\end{verbatim}
	
\noindent b) For $S$ and $T$ numerical semigroups, $g:S\to T$ a monoid morphism (and thus multiplication by an integer) and $E$ a relative ideal of $T$ with $E\subseteq T$, then $S\bowtie^gE$ is also a good semigroup (here $D=\{(s,ks) \mid s\in S\}$):
\[
S\bowtie^g E=D\cup(g^{-1}(E)\times E)\cup \{a\wedge b\mid a\in D, b\in
g^{-1}(E)\times E\},
\]
called the \emph{amalgamation} of $S$ with $T$ along $E$ with
respect to $g$, \cite{danna}.
\begin{verbatim}
gap> t:=NumericalSemigroup(3,4);;
gap> e:=3+t;
gap> a:=Amalgamation(s,e,2);
<Good semigroup>
gap> SmallElements(a);
[ [ 0, 0 ], [ 2, 3 ], [ 2, 4 ], [ 3, 3 ], [ 3, 6 ], [ 3, 7 ],
[ 3, 8 ], [ 3, 9 ], [ 4, 3 ], [ 4, 6 ], [ 4, 7 ], [ 4, 8 ],
[ 5, 3 ], [ 5, 6 ], [ 5, 7 ], [ 5, 9 ] ]
\end{verbatim}
These examples are illustrated in Figure \ref{fig:duplication}.

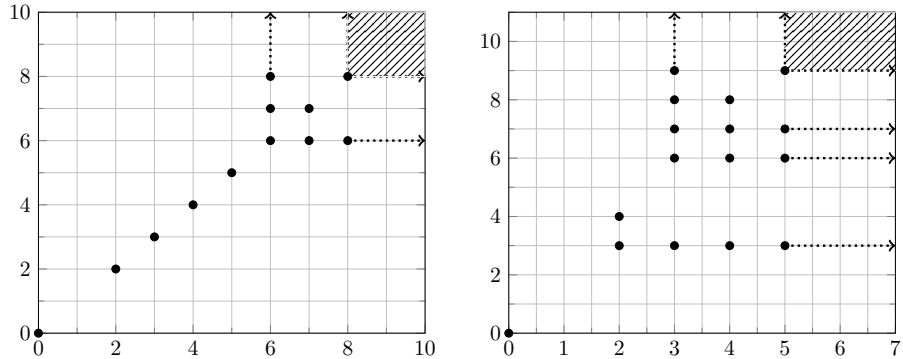
\begin{figure}[h]

\centering
\begin{tikzpicture}[scale=.75]% coordinates
\begin{axis}[grid=both, minor tick num=1, ymin=0, ymax=10,xmin=0, xmax=10,ytick={0,2,...,10}]
\addplot [only marks] coordinates {
	(0,0)
	(2,2)
	(3,3)
	(4,4)
	(5,5)
	(6,6)
	(6,7)
	(6,8)
	(7,6)
	(7,7)
	(8,6)
	(8,8)};
\addplot [->, style=dotted, very thick] coordinates{
	(6,8)
	(6,10)};
\addplot [->, style=dotted, very thick] coordinates{
	(8,6)
	(10,6)};
\addplot [->, style=dotted, very thick] coordinates{
	(8,8)
	(8,10)};
\addplot [->, style=dotted, very thick] coordinates{
	(8,8)
	(10,8)};
\addplot [pattern = north east lines,  draw=white] coordinates{
	(8,8)
	(10,8)
	(10,10)
	(8,10)
	(8,8)};
\end{axis}
\end{tikzpicture}\quad
\begin{tikzpicture}[scale=.75]
\begin{axis}[grid=both, minor tick num=1, ymin=0, ymax=11, xmin=0, xmax=7,ytick={0,2,...,11}]
% Plot code
\addplot[only marks] coordinates{
	(0, 0)
	(2, 3)
	(2, 4)
	(3, 3)
	(3, 6)
	(3, 7)
	(3, 8)
	(3, 9)
	%(3, 10)
	%(3, 11)
	%(3, 12)
	%(3, 13)
	(4, 3)
	(4, 6)
	(4, 7)
	(4, 8)
	(5, 3)
	(5, 6)
	(5, 7)
	(5, 9)
	%(5, 10)
	%(5, 11)
	%(5, 12)
	%(5, 13)
	%(6, 3)
	%(6, 6)
	%(6, 7)
	%(6, 9)
	%(6, 10)
	%(6, 11)
	%(6, 12)
	%(6, 13)
	%(7, 3)
	%(7, 6)
	%(7, 7)
	%(7, 9)
	%(7, 10)
	%(7, 11)
	%(7, 12)
	%(7, 13)
	%(8, 3)
	%(8, 6)
	%(8, 7)
	%(8, 9)
	%(8, 10)
	%(8, 11)
	%(8, 12)
	%(8, 13)
	%(9, 3)
	%(9, 6)
	%(9, 7)
	%(9, 9)
	%(9, 10)
	%(9, 11)
	%(9, 12)
	%(9, 13)
	};
\addplot [pattern = north east lines,  draw=white] coordinates{
	(5,11)
	(5,9)
	(9,9)
	(9,11)
	(5,11)};
\addplot [->, style=dotted, very thick] coordinates{
	(5,9)
	(5,11)};
\addplot [->, style=dotted, very thick] coordinates{
	(5,9)
	(7,9)};
\addplot [->, style=dotted, very thick] coordinates{
	(3,9)
	(3,11)};
\addplot [->,style=dotted, very thick] coordinates{
	(5,7)
	(7,7)};
\addplot [->,style=dotted, very thick] coordinates{
	(5,6)
	(7,6)};
\addplot [->,style=dotted, very thick] coordinates{
	(5,3)
	(7,3)};
%\draw (5,9) -- (5,13);
\end{axis}
\end{tikzpicture}
\caption{The good semigroups $\langle 2,3\rangle \bowtie (6+\langle 2,3\rangle)$ (left) and $\langle 2,3\rangle \bowtie^g (3+\langle 3,4\rangle)$ with $g$ multiplication by $2$ (right).}
\label{fig:duplication}
\end{figure}
	
\noindent c) According to \cite[Proposition 2.3]{a-u}, the direct product of two numerical semigroups is a good semigroup.

Let $S=\langle 3,5,7\rangle$ and $T=\langle 4,5\rangle$. Then $S\times T$ is a good semigroup.
	
\begin{verbatim}
gap> s:=NumericalSemigroup(3,5,7);;
gap> t:=NumericalSemigroup(4,5);;
gap> sms:=SmallElements(s);
[ 0, 3, 5 ]
gap> smt:=SmallElements(t);
[ 0, 4, 5, 8, 9, 10, 12 ]
gap> c:=Cartesian(sms,smt);;
gap> RepresentsSmallElementsOfGoodSemigroup(c);
true
\end{verbatim}
	
\noindent d) As we already mentioned, value semigroups are good semigroups. The value semigroup of the ring $\mathbb K[\![x,y]\!]/(y^4-2x^3y^2-4x^5y+x^6-x^7)(y^2-x^3)$ can be drawn as follows (see \cite[Figure 1]{apery-alg}).

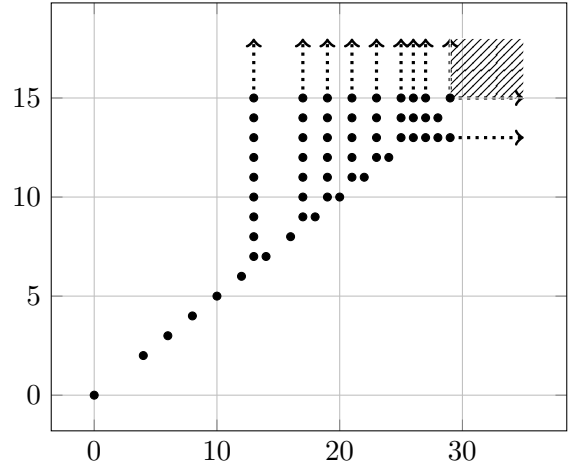
\medskip
\begin{center}
\begin{tikzpicture}
\begin{axis}[grid=both]
\addplot[->, style=dotted, very thick] coordinates{
	(29,13)
	(35,13)};
\addplot[->, style=dotted, very thick] coordinates{
	(29,15)
	(35,15)};
\addplot[->, style=dotted, very thick] coordinates{
	(13,15)
	(13,18)};
\addplot[->, style=dotted, very thick] coordinates{
	(17,15)
	(17,18)};
\addplot[->, style=dotted, very thick] coordinates{
	(19,15)
	(19,18)};
\addplot[->, style=dotted, very thick] coordinates{
	(21,15)
	(21,18)};
\addplot[->, style=dotted, very thick] coordinates{
	(23,15)
	(23,18)};
\addplot[->, style=dotted, very thick] coordinates{
	(25,15)
	(25,18)};
\addplot[->, style=dotted, very thick] coordinates{
	(26,15)
	(26,18)};
\addplot[->, style=dotted, very thick] coordinates{
	(27,15)
	(27,18)};
\addplot[->, style=dotted, very thick] coordinates{
	(29,15)
	(29,18)};
\addplot [pattern = north east lines,  draw=white] coordinates{
	(29,15)
	(35,15)
	(35,18)
	(29,18)
	(29,15)};
\addplot[only marks, mark options={scale=.75,solid}] coordinates{
	(0,0)
	(4,2)
	(6,3)
	(8,4)
	(10,5)
	(12,6)
	(13,7)
	(13,8)
	(13,9)
	(13,10)
	(13,11)
	(13,12)
	(13,13)
	(13,14)
	(13,15)
	(14,7)
	(16,8)
	(17,9)
	(17,10)
	(17,11)
	(17,12)
	(17,13)
	(17,14)
	(17,15)
	(18,9)
	(19,10)
	(19,11)
	(19,12)
	(19,13)
	(19,14)
	(19,15)
	(20,10)
	(21,11)
	(21,12)
	(21,13)
	(21,14)
	(21,15)
	(22,11)
	(23,12)
	(23,13)
	(23,14)
	(23,15)
	(24,12)
	(25,13)
	(25,14)
	(25,15)
	(26,13)
	(26,14)
	(26,15)
	(27,13)
	(27,14)
	(27,15)
	(28,13)
	(28,14)
	(29,13)
	(29,15)};
\end{axis}
\end{tikzpicture}
\end{center}
The value semigroup of $\mathbb K[\![x,y]\!]/(y^4-2x^3y^2-4x^5y+x^6-x^7)$ is $S_1=\langle 4,6,13\rangle$, while that of $K[\![x,y]\!]/(y^2-x^3)$ is $S_2=\langle 2,3\rangle$.
%\begin{verbatim}
%gap> f2;
%-x^7+x^6-4*x^5*y-2*x^3*y^2+y^4
%gap> app(f2,y,4);
%y
%gap> Resultant(f2,y,y);
%-x^7+x^6
%gap> app(f2,y,2);
%-x^3+y^2
%gap> Resultant(f2,last,y);
%x^14-16*x^13
%\end{verbatim}

Here the set of maximal elements mentioned in \cite{garcia} is
\begin{multline*}
M=\{ ( 0, 0 ), ( 4, 2 ), ( 6, 3 ), ( 8, 4 ), ( 10, 5 ), ( 12, 6 ), ( 14, 7 ),   ( 16, 8 ),\\ ( 18, 9 ), (20, 10 ), ( 24, 12 ), ( 22, 11 ), ( 28, 14 ) \},
\end{multline*}
and the good semigroup is the set of elements in $S_1\times S_2$ that are not above nor to the right an element of $M$ \cite[Theorem 6]{garcia}.
\begin{verbatim}
gap> s1:=NumericalSemigroup(4,6,13);;
gap> s2:=NumericalSemigroup(2,3);;
gap>M:=[[0,0],[4,2],[6,3],[8,4],[10,5],[12,6],[14,7],
>[16,8],[18,9],[20,10],[24,12],[22,11],[28,14]];;
gap> g:=GoodSemigrouByMaximalElements(s1,s2,M);
<Good semigroup>
\end{verbatim}

It is worth mentioning that semigroup duplication and
amalgamations can also be realized as value semigroups of
algebroid curves with two branches (see, for instance,
\cite[Section 2]{danna}).
\end{example}

From now on, we assume $S\subseteq \mathbb{N}^n$ to be a good semigroup and we denote $I=\{1,\dots,n\}$. Then
\[
H_J=\{(a_1,\dots,a_n)\in \mathbb Z^n\mid a_j\ge 0 \text{ for } j\in J, a_i=0 \text{ for } i\in I\setminus J\}
\]
is the ``semihyperplane'' in the nonnegative quadrant represented by the coordinates in $J$. In particular, when $J=\{j\}$, $H_J$ is the $j$-th semiaxis.
To make notations easier, when we make examples in the two branches case we call $OX=H_{\{1\}}$ and $OY=H_{\{2\}}$.
We also define, for every good relative ideal $E$ of $S$ and for every $J\subseteq I$, the \emph{$J$-border of $\mathrm{Small}(E)$} as
\[\partial_J(E)=\{a\in\mathrm{Small}(E)\mid a_j=\C(E)_j \text{ for all }j\in J\},\]
and the \emph{border} of $\mathrm{Small}(E)$ as \[\partial(E)=\bigcup_{J\subseteq I}\partial(E)_J.\]

Let $E$ a relative ideal of $S$, $a \in \mathbb Z^n$ and $J \subset I$.
Then we have the following technical definitions:
\begin{enumerate}[(1)]
	\item $\Delta_J (a):=\{b\in \mathbb Z^n \mid a_j=b_j \text{ for } j\in J \textup{ and } a_i<b_i \text{ for } i\in I\setminus J\}$,
	\item $\bar\Delta_J(a):=\{b\in \mathbb Z^n \mid a_j=b_j \text{ for } j\in J \textup{ and } a_i\le b_i \text{ for } i\in I\setminus J\}$,
	\item $\Delta_J^E (a) := \Delta_J (a) \cap E$,
	\item $\bar\Delta_J^E (a) := \bar\Delta_J (a) \cap E$,
	\item $\Delta(a):=\bigcup_{i\in I} \Delta_i (a)$, where $\Delta_i(a):=\Delta_{\{i\}}(a)$,
	\item $\bar\Delta(a):=\bigcup_{i\in I} \bar\Delta_i (a)$, where $\bar\Delta_i(a):=\bar\Delta_{\{i\}}(a)$,
	\item $\Delta^E (a) := \Delta (a) \cap E$,
	\item $\bar\Delta^E (a) := \bar\Delta (a) \cap E$.
\end{enumerate}
Notice that $H_J$ plays the dual role of $\bar\Delta_J$, meaning that 
\[
(a+H_J)=\{b\in\mathbb Z^n\mid b_j\ge a_j \text{ for } j\in J, b_i=a_i \text{ for }i\in I\setminus J\}=\bar\Delta_{I\setminus J}(a).
\]

The following Lemma has already been proved, in a slightly different way, in \cite[Lemma 4.1.11]{K-T-S}. We include it here for sake of completeness.
\begin{lemma}\label{lem:contains-border-axes}
Let $E$ be a good relative ideal and $a\in E$. If $a\in\partial_J(E)$ for some $J\subseteq I$, then $(a+H_J)\subseteq E$. 
\end{lemma}
\begin{proof}
Let $c \in (a+H_J)$. Then, by the definition of $H_J$ and $\partial_J(E)$:
\begin{align*}
c_j \ge&  a_j=\C(E)_j \text{ for all } j \in J, \qquad\qquad \\
c_i = & a_i  \text{ for all } i \in I \setminus J.
\end{align*}
Let now $b\in\mathbb Z^n$ be such that
\begin{align*}
b_j =& a_j \text{ for all } j \in J, \\
b_i >& \max\{\C(E)_i, a_i \}  \text{ for all } i \in I \setminus J.
\end{align*}
Then $b \geq \C(E)$, which implies $b \in E$.
Now applying property (G2) to $a$ and $b$ we obtain, for any $j \in J$, an element $a' \in E$ with $a' \geq a + \textbf{e}_j$, where $\textbf e_i=(0,\dots,\underset{i}{1},\dots,0)$. Therefore, repeating the process substituting $a$ with $a'$ and taking again a $b$ with the above properties, we obtain an element $\overline a$ such that 
\begin{align*}
\overline a_j >& a^{(n)}_j\ge a_j \text{ for all } j \in J, \\ 
\overline a_i =& b_i\wedge\max\{\C(E)_i, a_i \}\}=\max\{\C(E)_i, a_i\}\ge c_i \text{ for all } i \in I \setminus J.
\end{align*}
For $n$ big enough, we can suppose $\overline a\ge c$.
	
Pick now $d \in \mathbb Z^n$ such that
\begin{align*}
d_j =& c_j\ge \C(E)_j \text{ for all } j \in J, \\
d_i >& \max\{\C(E)_i, c_i\} \text{ for all }i \in I \setminus J.
\end{align*}
In particular, $d \geq \C(E)$, and hence $d \in E$. 
Thus, $c=d\wedge{a}\in E$ since $E$ satisfies (G1).
\end{proof}

Once we know $\C(E)$ and $\mathrm{Small}(E)$ we can easily check membership to $E$.
\begin{proposition}\label{prop:membership}
Let $a\in\mathbb N^n$. Then $a\in E$ if and only if $a\wedge \C(E)\in\mathrm{Small}(E)$.
\end{proposition}
\begin{proof}
If $a\ge\C(E)$, then clearly $a\in E$, by definition of conductor. On the other hand, if $a\le\C(E)$ then $a=a\wedge \C(E)\in\mathrm{Small}(E)$ implies $a\in E$.
If none of the two is the case, then let $b=a\wedge \C(E)$. Then $b\in\partial_J(E)$ for some $J\subseteq I$ and $a\in (b+H_J)$. By Lemma \ref{lem:contains-border-axes}, this gives $a\in E$.
\end{proof}

From the last result it easily follows the essentially well-known fact that a good semigroup is fully determined by its small elements.
\begin{corollary}\label{smallelements}
Let $S$ and $S^\prime$ be two good semigroups. Then $S=S^\prime$ if and only if $\mathrm{Small}(S)=\mathrm{Small}(S^\prime)$.
\end{corollary}

In the same way, fixed a semigroup $S$, its good relative ideals are fully determined by their small elements.

\begin{corollary}\label{smallelementsforideals}
Let $S$ be a good semigroup. Let $E$ and $E'$ be two good relative ideals of $S$ with $\C(E)=\C(E')$. Then $E=E'$ if and only if $\mathrm{Small}(E)=\mathrm{Small}(E^\prime)$.
\end{corollary}

Summarizing, for a good relative ideal $E$ of $S$ we get:
\[
E=\mathrm{Small}(E)\cup(\C(E)+\mathbb N^n)\cup \left(\bigcup\nolimits_{a\in\partial_J(E), J\subseteq I}(a+H_J)\right),
\]
and the same of course holds for the semigroup $S$.
Notice that this notation is redundant, since if $J^\prime\subseteq J\subseteq I$ and $a\in \partial_J(E)$, then $a\in \partial_{J^\prime}(E)$.

\section{Good generating systems for semigroups}\label{sec:good-gen-sys}

We now concentrate on good semigroups. We will analyse good relative ideals in the next section.
First of all, observe that intersection of good semigroups might fail to be good, as the following example shows.
\begin{example}\label{ex-not-so-good}
Let $S$ and $T$ be two good semigroups with conductor $\C=(5,5)$ and with  $\mathrm{Small}(S)=\{(0,0),(3,3),(3,4),(4,3)\}$ and $\mathrm{Small}(T)=\{(0,0),(3,3),(3,5),(4,3)\}$. Then  $S\cap T$ is not a good semigroup since condition (G2) does not hold.
\end{example}

This has bad consequences: for a subset $G$ of $\mathbb N^n$ we
cannot think about the intersection of all good semigroups
containing $G$ as the good semigroup generated by $G$. Hence we
have to look for a convenient alternative.

Good semigroups are submonoids of $\mathbb N^n$, and so they have
a unique minimal generating system: if $S$ is a submonoid of
$\mathbb N^n$, then its minimal generating system is $S^*\setminus
(S^*+S^*)$, with $S^*=S\setminus\{(0,0)\}$ (see, for instance,
\cite{fg}).
Unfortunately, good semigroups are not in general
affine semigroups, that is, they are not finitely generated as
monoids. For instance for $S=\langle 2,3\rangle\bowtie (6+\langle
2,3\rangle)$ in Example \ref{ex:two-good}, we need $(8,6)+OX$ to
be in the minimal generating set of $S$.

We know that a good semigroup $S$ is fully determined by
$\mathrm{Small}(S)$, which is a finite set. Among the elements of
$\mathrm{Small}(S)$, some might be linear combinations or infima
of others, and so in many cases we could choose a smallest subset
of $\mathrm{Small}(S)$ that still describes uniquely $S$.  To this
end, for $G\subseteq\mathbb N^n$, define $[G]$ to be the smallest
submonoid of $\mathbb N^n$ containing $G$ that is closed under addition and
infima.

For a subset $A$ of a monoid $M$, we denote by
\[
\langle A\rangle=\{a_1+\cdots + a_m\mid m\in \mathbb N, a_1,\dots,a_m\in A\}
\]
the submonoid of $M$ generated by $A$.

\begin{proposition}\label{prop:dist}
Let $G$ be a subset of $\mathbb N^n$. Then
\[
[G]=\{g_1\wedge\dots\wedge g_m\mid m\in\mathbb N, g_1,\ldots, g_m\in \langle G\rangle\}.
\]
\end{proposition}
\begin{proof}
Observe that $+$ distributes with respect to $\wedge$, that is, for $a,b,c\in \mathbb N^n$,
\begin{equation}\label{eq:associative}
(a\wedge b)+c=(a+c)\wedge(b+c).
\end{equation}
The proof now follows from  \eqref{eq:associative} and the fact
that intersections of submonoids of $(\mathbb N^n,+)$ closed under
infima are  again submonoids of $(\mathbb N^n,+)$ closed under
infima.
\end{proof}

Notice that for every set $A=\{a_1,\dots,a_m\}\subset \mathbb N^n$ the infimum of $A$, i.e.~$a_1\wedge\dots\wedge a_m$, is the infima of at most $n$ elements of $A$. Thus, we get the following consequence.

\begin{corollary}\label{cor:dist}
Let $G\subseteq\mathbb N^n$. Then 
\[
[G]=\{g_1\wedge\dots\wedge g_n\mid g_i\in \langle G\rangle\}.
\]
\end{corollary}

Given a $d\in \mathbb N^n$, we define 
\[
[G]_d:=d\wedge[G]
\]
and
\[
\mathrm{B}(d)=d\wedge\mathbb N^n=\{a\in\mathbb N^n\mid a\le d\}.
\]

As we mentioned above, we are mostly interested in when $[G]$ covers completely $\mathrm{Small}(S)$ for some $G\subseteq S$. 

If $\C$ is the conductor of $S$, we have $\mathrm{Small}(S)=\C\wedge S$ by definition of $\mathrm{Small}(S)$.
We therefore say that $G$ is a \emph{good generating system} for $S$ if 
\[
[G]_\C=\C\wedge [G]=\mathrm{Small}(S).
\]
We will say that $G$ is \emph{minimal} if no proper subset of $G$ is a good generating system of $S$.

Observe that $[G]_\C=[\C\wedge G]_\C$, thus we will always assume that good systems of generators are contained in $\mathrm B(\C)$.

Clearly $[\mathrm{Small}(S)]_\C=\mathrm{Small}(S)$, so that $\mathrm{Small}(S)$ is always a good generating system. However, it doesn't need to be minimal. We would like, when possible, to remove redundant elements of $\mathrm{Small}(S)$. The following trivial result is the first natural filter.
\begin{lemma}\label{lemma:combination}
Let $S$ be a good semigroup. Let $G$ be a good generating system for $S$ and let $a\in G$. If $a\in[G\setminus\{a\}]_\C$, then $G\setminus\{a\}$ is a good generating system for $S$.
\end{lemma}

Given a good system of generators, it is helpful to have a procedure to determine if an element is in the semigroup. We describe a method in the following two lemmas. 

\begin{lemma}\label{membership-semiring}
Let $G\subseteq\mathbb N^n$ and $a\in\mathbb N^n$. Then $a\in[G]$ if and only if $\bar\Delta_i(a)\cap\langle G\rangle\not=\emptyset$ for all $i\in I$. 
\end{lemma}
\begin{proof}
Suppose $a\in[G]$. If $a\in\langle G\rangle$, then $a\in\bar\Delta_i(a)\cap\langle G\rangle$ for all $i\in I$.
So let us suppose $a\in[G]\setminus\langle G\rangle$. By Corollary \ref{cor:dist} we know that $a\in[G]$ if and only if it is the minimum of $n$ elements of $\langle G\rangle$, $\{b^{(i)}\}_{i\in I}$, and since we are assuming $a\not\in \langle G\rangle$ we have $a\not=b^{(i)}$ for all $i\in I$. Since $a$ is the minimum of the elements $b$, we have $a_j\le b^{(i)}_j$ for all $j\in I$. Moreover, for any $j\in I$ there exists an $i_j\in I$ such that $a_j=b^{(i_j)}_j$. This means that for each $i\in I$ there is an index $i_j$ with $b^{(i_j)}$ belonging to $\bar\Delta_i(a)$. In particular, for any $i\in I$, $b^{(i_j)}\in \bar\Delta_i(a)\cap \langle G\rangle$. 
	
For the converse, suppose $\bar\Delta_i(a)\cap\langle G\rangle\not=\emptyset$ for all $i\in I$. If $a\in \langle G\rangle$, then trivially $a\in [G]$. Hence assume that $a\not\in \langle G\rangle$, and let for each $i\in I$, $b^{(i)}\in \bar\Delta_i(a)\cap\langle G\rangle$. Then clearly $a=~\wedge_{i\in I} b^{(i)}\in [G]$.
\end{proof}

We can be more precise if we substitute $[G]$ with $[G]_d$ for some $d$.
\begin{lemma}\label{membership-good-system}
Let $d\in\mathbb N^n$ and $G\subseteq \mathrm B(d)$. Let $a\in \mathrm B(d)\setminus\{d\}$ and let $J$ be maximal (with respect to inclusion) with the property $d\in \Delta_J(a)$ ($J$ can also be empty). Then $a\in [G]_d$ if and only if $\bar\Delta_i(a)\cap\langle G\rangle\not=\emptyset$ for all $i\in I\setminus J$.
\end{lemma}
\begin{proof}
Suppose $a\in[G]_d$. Then $a=a'\wedge d$ for some $a'\in[G]$. By Lemma \ref{membership-semiring}, $\bar\Delta_i(a')\cap\langle G\rangle\not=\emptyset$ for all $i\in I$. From the assumptions we have $a_j=a'_j$ for $j\in J$ and $a_i<a'_i$ for $i \in I\setminus J$. Therefore $a'$ has the following properties: $a'_j\ge a_j$ for $j\in J$ and $a'_i=a_i$ for $i\in I\setminus J$. So $\bar\Delta_i(a')\subset\bar\Delta_i(a)$ for any $i\in I\setminus J$, and so $\bar\Delta_i(a)\cap\langle G\rangle\not=\emptyset$ for all $i\in I\setminus J$.
	
Vice versa, suppose $\bar\Delta_i(a)\cap\langle G\rangle\not=\emptyset$ for all $i\in I\setminus J$. Let $b^{(i)}\in \bar\Delta_i(a)\cap\langle G\rangle$. Then $a=d\wedge (\wedge_{i\in I} b^{(i)})\in [G]_d$.
\end{proof}

We can potentially remove more elements. We will first focus on the local case.

\subsection{Local case}

A good semigroup $S$ is said to be \emph{local} if the only element of $S$ in the axes is $0$. Hence good systems of generators will not have elements in the axes.

It can be shown that every good semigroup is a direct product of good local semigroups \cite[Theorem 2.5]{a-u}. When considering the two branches setting, unless $S$ is the direct product of two numerical semigroups, then $S$ is local.

We devote this section to characterizing good minimal generating systems for local good semigroups.

\begin{lemma}\label{lem:shared-coord}
Let $S$ be a good local semigroup. Let $G$ be a good generating system for $S$ and $a\in G\setminus\partial(S)$. If there exists $b\in \bar\Delta(a)\cap\langle G\setminus\{a\}\rangle$, then $G\setminus\{a\}$ is a good generating system for $S$.
\end{lemma}
\begin{proof}
Since $a\in G\setminus\partial(S)\subset \mathrm{Small}(S)\setminus\partial(S)$, we have $a_i<\C_i$ for all $i\in I$. 
If $a=b<\C$, then we are done, since then clearly $a\in[G\setminus\{a\}]_C$. 
	
So let us suppose $a\not=b$. Since $b\in \bar\Delta(a)$, there exists an $i\in I$ such that $b_i=a_i$. Then, applying (G2), one can find  a $c\in S\setminus\{a,b\}$ such that $a=b\wedge c$. 
Eventually substituting $c$ with $c\wedge \C$ we can assume $c\in \mathrm{Small}(S)=[G]_\C$. Since $a\not\in\partial(S)$ after this substitution we still have $c\not=a\not=b$. 
	
By Corollary \ref{cor:dist}, we can write
\[
c=\C\wedge(\wedge_{i\in I}c^{(i)}),
\]
for some $\{c^{(i)}\}_{i\in I} \subset \langle G\rangle$. As $c\not=a$ and $a<c\le c^{(i)}$, we also have $a\not=c^{(i)}$ for any $i\in I$.
Let $J\subseteq I$ be the maximal set of indices such that $c\in\bar\Delta_J(a)$ (which implies $b\in\bar\Delta_{I\setminus J}(a)$). Then for every $j\in J$ there exists an $i_j\in I$ such that $c^{(i_j)}_j=c_j=a_j$.
Hence for any $j\in J$, $c^{(i_j)}\in\bar\Delta_j(a)$ and moreover 
\[
a=b\wedge(\wedge_{j\in J} c^{(i_j)}).
\] 
Since $c^{(i_j)}\in\langle G\rangle$, for any $j\in J$ we can write
\[
c^{(i_j)}=d^{(i_j)}_1+\cdots+d^{(i_j)}_{k_{i_j}}
\]
where $\{d^{(i_j)}_h\}_{h\in\{1,\dots,k_{i_j}\}}\subset G$ for any $j\in J$. This yields
\[
c^{(i_j)}=d^{(i_j)}_1+\cdots+d^{(i_j)}_{k_{i_j}}\not=a
\]
and
\[
c^{(i_j)}_j=a_j.
\]
As $S$ is local, this implies $d^{(i_j)}_h\not=a$ for any $h\in\{1,\dots,k_{i_j}\}$ and $j\in J$. Therefore, $c^{(i_j)}\in\langle G\setminus \{a\}\rangle$ for any $j\in J$ and $a=b\wedge(\wedge_{j\in J} c^{(i_j)})\in[G\setminus\{a\}]_\C$.
\end{proof}

In particular, if two elements in a minimal good system of generators share a coordinate, then they must be in the border corresponding to that coordinate.

Let $a\in \mathbb N^n$. We say that $a$ is \emph{positive} if $a$ is not in the axes.

\begin{remark}\label{cond-not-gen}
Let $S$ be a good semigroup with conductor $\C$. Observe that if a good generating system $G$ contains a positive element $a$, then there is a positive integer $k$ such that $\C\le k a$.
Hence $\C\in [G]_\C$. We can then assume that, unless $G=\{\C\}$, the conductor is never in a good generating system.
\end{remark}

We can still sharpen a bit more the characterization of minimal good generating systems, and use just affine spanning without infima.

\begin{theorem}\label{th:carac-minimality}
Let $S$ be a good local semigroup and let $G$ be a good generating system for $S$. For $a\in G$, let $J_a$ be the set of indices maximal with the property $a\in\partial_{J_a}(S)$. Then $G$ is a minimal good generating system if and only if for all $a\in G$, 
\[
\bar\Delta(a)\cap\langle G\setminus\{a\}\rangle=\emptyset, \text{ if } J_a=\emptyset,
\]
and there exists an $i\in I\setminus J_a$ such that
\[
\bar\Delta_i(a)\cap\langle G\setminus\{a\}\rangle=\emptyset, \text{ if } J_a\ne\emptyset.
\]
\end{theorem}
\begin{proof}
In order to simplify notation, if there is no possible misunderstanding with $a$, let us write $J$ instead of $J_a$. 

\noindent \emph{Necessity.} Assume that $G$ is a minimal good generating system for $S$ and let $a\in G$. If $J=\emptyset$, that is, $a\in G\setminus\partial(S)$, then the proof follows by Lemma \ref{lem:shared-coord}. 

Now assume that $J$ is not empty. Suppose to the contrary that for all $i\in I\setminus J$ there exists an $a^{(i)}\in\bar\Delta_i(a)\cap \langle G\setminus\{a\}\rangle$. 
Then $a=\C\wedge (\wedge_{i\in I\setminus J}a^{(i)})$, and consequently $a\in[G\setminus\{a\}]_\C$, which is a contradiction.
	
\noindent \emph{Sufficiency.} If $G=\{a\}$, then $G$ is minimal. Suppose therefore $|G|\ge2$ and $G$ not minimal. 
Then there exists an $a\in G$ such that $a\in[G\setminus\{a\}]_\C$. By Corollary $\ref{cor:dist}$, there exist $\{a^{(i)}\}_{i\in I}\subset \langle G\setminus\{a\}\rangle$ such that
\[
a=\C \wedge (\wedge_{i\in I}a^{(i)}).
\]
Since $S$ is local, $a$ is a positive element in $\mathbb N^n$. 
By Remark \ref{cond-not-gen}, it follows $a\not=\C$. 
Let $J$ be the set of indices maximal with the property $a\in\partial_{J}(S)$. 
Then $a_k<\C_k$ for any $k\in I\setminus J$, and since $a\neq C$, we have $J\neq I$.
Hence for any $k\in I\setminus J$ there is an $i_k$ with $a^{(i_k)}_k=a_k$ and $a^{(i_k)}\ge a$, that is, $a^{(i_k)}\in\bar\Delta_k(a)\cap\langle G\setminus\{a\}\rangle$ for each $k\in I\setminus J$.
This is a contradiction.
\end{proof}

Minimal good generating systems for good local semigroups are unique.

\begin{theorem}\label{th:unique}
Let $S$ be a good local semigroup. Then $S$ has a unique minimal good generating system.
\end{theorem}
\begin{proof}
Let $A$ and $B$ be two minimal good generating systems for $S$, $C$ the conductor of $S$ and let $b$ be minimal in $(A\cup B)\setminus(A\cap B)$.
Without loss of generality, we can assume $b\in B$.

Let us prove that 
%$b\not\in\langle A\rangle$. 
%If $b=\sum_la^{(l)}$ for some $\{a^{(l)}\}_l\subseteq A$.
%Then $a^{(l)}\le b$ for all $l$.
%As $b\not\in A$, we have $a^{(l)}\not=b$ for all $l$. 
%Consequently $a^{(l)}<b$ for all $l$. 
%The minimality of $b$ implies then $a^{(l)}\in A\cap B$, and thus $b\in\langle B\setminus\{b\}\rangle$. 
%This contradicts the minimality of $B$.
%Hence 
%\[
%b\not\in\langle A\rangle.
%\]
%Actually, something more general is also true: 
$b\not\in C\wedge\langle A\rangle$.
Assume that $b=C\wedge \left(\sum_l a^{(l)}\right)$ with $a^{(l)}\in A$.
As $b\not\in A$, the sum has more than one term.
Otherwise, $b=C\wedge a^{(1)}=a^{(1)}\in A$ which is a contradiction.
In particular, $a^{(l)}\ne b$ for all $l$.
As $a^{(l)}\le \sum_l a^{(l)}$ and $a^{(l)}\le C$, we have
$a^{(l)}\le C\wedge\left(\sum_l a^{(l)}\right)=b$.
Together with the considerations above, this gives $a^{(l)}<b$ for all $l$.
But then $a^{(l)}\in B$ by minimality of $b$ and thus $b\in[B\setminus\{b\}]_C$, which contradicts the minimality of $B$.
Thus 
\begin{equation}\label{511a}
b\not\in C\wedge\langle A\rangle.
\end{equation}

Let now $J\subset I$ be maximal such that $b\in \partial_J(S)$ ($J$ can also be empty).
As $b\in B\subseteq [A]_C$, by Lemma \ref{membership-good-system} there exist 
\[
\{x^{(i)}\}_{i\in I\setminus J}\subseteq \bar\Delta_i(b)\cap \langle A\rangle
\] 
such that
\[
b=C\wedge (\wedge_{i\in I\setminus J} x^{(i)}) \text{ with } x^{(i)}_i=b_i \text{ for } i\in I\setminus J.
\]
As $x^{(i)}$ do not need to be in $[B]_C$, let us consider $y^{(i)}=C\wedge x^{(i)}$ for all $i\in I\setminus J$.
Then
\[
\{y^{(i)}\}_{i\in I\setminus J}\subseteq \bar\Delta_i(b)\cap C\wedge \langle A\rangle
\]
and
\begin{equation}\label{511b}
b=C\wedge (\wedge_{i\in I\setminus J}y^{(i)}) \text{ with } y^{(i)}_i=b_i \text{ for } i\in I\setminus J.
\end{equation}
Let $K_i\subset I$ be the maximum set of indices such that $y^{(i)}\in \partial_{K_i}(S)$ for all $i\in I\setminus J$.
As $y^{(i)}\in [A]_C=[B]_C$, again by Lemma \ref{membership-good-system} there exist 
\[
\{d^{(i,j)}\}_{j\in I\setminus K_i}\subseteq \bar\Delta_j(y^{(i)})\cap\langle B\rangle
\]
such that
\[
y^{(i)}=C\wedge (\wedge_{j\in I\setminus K_i}d^{(i,j)}) \text{ with } d^{(i,j)}_j=y_j \text{ for } j\in I\setminus K_i.
\]
for all $i\in I\setminus J$.
Since $y^{(i)}_i=b_i<C_i$ (i.e. $i\not\in K_i$) for all $i \in I\setminus J$, and $d^{(i,j)}\ge y^{(i)}\ge b$, for all $i \in I\setminus J$ there exists $j_i$ such that 
\[
d^{(i,j_i)}\in \bar\Delta_i(y^{(i)})\cap \langle B\rangle\subseteq\bar\Delta_i(b)\cap\langle B\rangle.
\]
So by \eqref{511b} we can write
\[
b=C\wedge (\wedge_{i\in I\setminus J}d^{(i,j_i)}) \text{ with } d^{(i,j_i)}_i=b_i \text{ for } i\in I\setminus J.
\]
By Theorem \ref{th:carac-minimality}, this implies that there exists (at least) a $i\in I\setminus J$ such that 
\[
d:=d^{(i_{j_i})}\in\langle B\rangle\setminus \langle B\setminus\{b\}\rangle.
\] 
This means $d=b+z$, with $z\in \langle B\rangle$. 
Since $d_i=b_i$, $z_i=0$. 
But $S$ is local, and this forces $z=0$. 
So $d=b$. 
But by \eqref{511a}, $y^{(i)}\ne b$ for all $i\in I\setminus J$, so in particular $d\ge y^{(i)}>b$.
This is a contradiction.
So the claim is proved.
\end{proof}

Following \cite{garcia}, we call an element
$a\in\mathrm{Small}(S)$ a \emph{maximal element} for $S$ if
$\Delta^S(a)=\emptyset$. As a consequence of Theorem \ref{th:carac-minimality}, an element that is both in a minimal good generating system of $S$ and in $\mathrm{Small}(S)\setminus\partial(S)$ is a maximal element. 
%\todo[inline]{is it really clear?}
\begin{remark}\label{garc}
Let us consider the two branches case. In \cite[Theorem 6]{garcia} it is proved that, for a value
semigroup of a plane singularity with two branches, the maximal
elements determine the semigroup, once you know the two
projections. More precisely, denote by $S_i$, $i\in\{1,2\}$, the
projections of the good semigroup $S$ in the $x$ and $y$ axes,
respectively. Let $T=\{(x,y)\ |\ x\in S_1,\ y\in S_2,\ (x,y)\notin
\Delta(a)\setminus \{a\}, \textrm{ for all } a\ \text{maximal}\}$.
By definition of maximal element, it follows that $S\subseteq T$.
Now take  any element $(x,y)\in T$. After a case by case checking,
that depends on the fact that both $x$ and $y$ can be the
components either of a maximal point or of a element in the border
of $S$, it is straightforward to see that either $(x,y)$ is a
maximal point or it can be obtained as an infimum of elements in
$S$. Hence $(x,y)\in S$.

However, even if in the minimal generating set of $S$ there are
elements in the border, we are requesting a smaller set of data
with respect to \cite[Theorem 6]{garcia}, since to reconstruct $S$
from the maximal elements one also needs to know the two
projections.

Notice also that if we start from a minimal system of generators,
to reconstruct the semigroup we need to know the conductor or, at
least, that the conductor has to be smaller than or equal to a
given element of $\mathbb N^n$ (see the next example and Remark
\ref{C-1}).
\end{remark}

\begin{example}
Let $S$ be the value semigroup of the ring $\mathbb K[\![x,y]\!]/(y^4-2x^3y^2-4x^5y+x^6-x^7)(y^2-x^3)$ . Then the maximal elements of $S$ are
\begin{multline*}
\{ ( 0, 0 ), ( 4, 2 ), ( 6, 3 ), ( 8, 4 ), ( 10, 5 ), ( 12, 6 ), ( 14, 7 ),   ( 16, 8 ), ( 18, 9 ), ( 20, 10 ), \\ ( 24, 12 ), ( 22, 11 ), ( 28, 14 ) \},
\end{multline*}
while a minimal good generating system is
\[
\{ ( 4, 2 ), ( 6, 3 ), (13, 15 ), ( 26, 15 ), ( 29, 13 ) \}.
\]
Moreover, in order to describe $S$ by the maximal elements, we
also need the value semigroups of $\mathbb
K[\![x,y]\!]/(y^4-2x^3y^2-4x^5y+x^6-x^7)$ and $\mathbb
K[\![x,y]\!]/(y^2-x^3)$, that is, $\langle 4,6,13\rangle$ and
$\langle 2,3\rangle$, respectively; on the other
hand, to describe $S$ by the minimal generating system we need to
know the the conductor $(29,15)$.
\end{example}

Maximal elements of good semigroups may not behave as in
value semigroups of planar curves with two branches, as the
following example shows. This can be used to produce good semigroups that are not value semigroups of curves.
%\todo[inline]{why is this relevant in our context? A reason should be added, at the moment looks just like a fun fact, since we do not really work with maximal elements}
\begin{example}
It is well known (see for instance \cite{garcia}) that there is a
symmetry in the set of maximal elements of the value semigroup of
a planar curve with two branches. This symmetry is expressed in the following:
if $\C$ is the conductor of the value semigroup, and $a$ is a
maximal element, then so is $\C-(1,1)-a$. This is not the case
in general for any good semigroup. Take for instance the
\cite[Example 2.16]{a-u}.
\begin{center}
\begin{tikzpicture}
 \begin{axis}[grid=both, xmin=0,xmax=32]
\addplot[->, style=dotted, very thick] coordinates{
(25,12)
(30,12)};
\addplot[->, style=dotted, very thick] coordinates{
(25,15)
(30,15)};
\addplot[->, style=dotted, very thick] coordinates{
(25,16)
(30,16)};
\addplot[->, style=dotted, very thick] coordinates{
(25,18)
(30,18)};
\addplot[->, style=dotted, very thick] coordinates{
(25,19)
(30,19)};
\addplot[->, style=dotted, very thick] coordinates{
(25,21)
(30,21)};
\addplot[->, style=dotted, very thick] coordinates{
(25,22)
(30,22)};
\addplot[->, style=dotted, very thick] coordinates{
(25,24)
(30,24)};
\addplot[->, style=dotted, very thick] coordinates{
(25,25)
(30,25)};
\addplot[->, style=dotted, very thick] coordinates{
(25,27)
(30,27)};
\addplot[->, style=dotted, very thick] coordinates{
(14,27)
(14,33)};
\addplot[->, style=dotted, very thick] coordinates{
(15,27)
(15,33)};
\addplot[->, style=dotted, very thick] coordinates{
(18,27)
(18,33)};
\addplot[->, style=dotted, very thick] coordinates{
(19,27)
(19,33)};
\addplot[->, style=dotted, very thick] coordinates{
(21,27)
(21,33)};
\addplot[->, style=dotted, very thick] coordinates{
(22,27)
(22,33)};
\addplot[->, style=dotted, very thick] coordinates{
(23,27)
(23,33)};
\addplot[->, style=dotted, very thick] coordinates{
(25,27)
(25,33)};
\addplot [pattern = north east lines,  draw=white] coordinates{
(25,27)
(30,27)
(30,33)
(25,33)
(25,27)
};
\addplot[only marks, mark options={scale=.5,solid} ] coordinates{
(0,0)
(4,3)
(7,6)
(7,9)
(7,12)
(7,13)
(8,6)
(11,9)
(11,12)
(11,15)
(11,16)
(11,17)
(12,9)
(14,12)
(14,15)
(14,16)
(14,18)
(14,19)
(14,20)
(14,21)
(14,22)
(14,23)
(14,24)
(14,25)
(14,26)
(14,27)
(15,12)
(15,15)
(15,16)
(15,18)
(15,19)
(15,20)
(15,21)
(15,22)
(15,23)
(15,24)
(15,25)
(15,26)
(15,27)
(16,12)
(16,15)
(16,16)
(16,18)
(16,19)
(16,20)
(18,12)
(18,15)
(18,16)
(18,18)
(18,19)
(18,21)
(18,22)
(18,23)
(18,24)
(18,25)
(18,26)
(18,27)
(19,12)
(19,15)
(19,16)
(19,18)
(19,19)
(19,21)
(19,22)
(19,23)
(19,24)
(19,25)
(19,26)
(19,27)
(20,12)
(20,15)
(20,16)
(20,18)
(20,19)
(20,21)
(20,22)
(20,23)
(21,12)
(21,15)
(21,16)
(21,18)
(21,19)
(21,21)
(21,22)
(21,24)
(21,25)
(21,26)
(21,27)
(22,12)
(22,15)
(22,16)
(22,18)
(22,19)
(22,21)
(22,22)
(22,24)
(22,25)
(22,26)
(22,27)
(23,12)
(23,15)
(23,16)
(23,18)
(23,19)
(23,21)
(23,22)
(23,24)
(23,25)
(23,26)
(23,27)
(24,12)
(24,15)
(24,16)
(24,18)
(24,19)
(24,21)
(24,22)
(24,24)
(24,25)
(24,26)
(25,12)
(25,15)
(25,16)
(25,18)
(25,19)
(25,21)
(25,22)
(25,24)
(25,25)
(25,27)
};

\end{axis}
\end{tikzpicture}
\end{center}
For this good semigroup, the set of maximal elements can be computed as follows.
\begin{verbatim}
gap> G:=[[4,3],[7,13],[11,17],[14,27],[15,27],[16,20],[25,12],[25,16]];;
gap> g:=GoodSemigroup(G,[25,27]);
<Good semigroup>
gap> Conductor(g);
[ 25, 27 ]
gap> MaximalElementsOfGoodSemigroup(g);
[ [ 0, 0 ], [ 4, 3 ], [ 7, 13 ], [ 8, 6 ], [ 11, 17 ], [ 12, 9 ], [ 16, 20 ],
  [ 20, 23 ], [ 24, 26 ] ]
\end{verbatim}
Observe that $(24,26)-(7,13)$ is not a maximal element.

%This example also shows that w-generators introduced in \cite{c-d-gz} cannot be used to describe a good semigroup in general. \todo[inline]{I believe this is not true. Every good semigroup is $w$-generated, only not uniquely. By Statement 1 at the end of their paper, they allow more generators of type $w_1+(1,0)\mathbb N$. The restriction to only 2 $w$-generators should be true only in the plane case.} Agree, when I wrote this, Marco warned me that at the end of their paper Félix et al give a more general situation... It is true that it cannot be described as in their Theorem 1. But I agree that it is not worth entering this discussion.
%A set of w-generators for a good semigroup $S$ in $\mathbb N^2$ is a set of the form $W=A\cup (w_1+(1,0)\mathbb N)\cup (w_2+(0,1)\mathbb N)$, with $A$ a finite subset of $S$, such that $W$ generates $S$ as a monoid. Notice that in our example $(14,27)+ (1,0)\mathbb N\subseteq S$ and also $(15,27)+(1,0)\mathbb N\subseteq S$. The leftmost vertical half line included in $S$ contains $(14,27)$, and thus $w_1$ should be in this half line. This implies that in order to get $(15,27)+(1,0)\mathbb N$ in the monoid generated by $W$, there should be an element of the form $(1,y)$ in the semigroup, which is not the case. Thus, such a system of generators does not exist for our semigroup $S$.
\end{example}

Notice that, if we have a set $G$ not fulfilling the conditions on Theorem \ref{th:carac-minimality} and a positive element $\C$, then $[G]_\C$ in general does not represent the set of small elements of a good semigroup.

\begin{example}
Let  $G=\{(2,2),(4,2)\}$ and $\C=(6,6)$. Then $[G]_\C$ looks like:
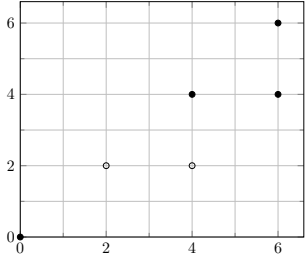
\begin{multicols}{2}
\begin{tikzpicture}[scale=.55]
\begin{axis}[grid=both, minor tick num=1, xmin=0, ymin=0, xtick={0,2,...,6}, ytick={0,2,...,6}]
\addplot [mark=o, only marks] coordinates {
(2,2)
(4,2)
};

\addplot [only marks] coordinates{
(0,0)
(4,4)
(6,4)
(6,6)
};
\end{axis}
\end{tikzpicture}
\columnbreak

Condition (G2) does not hold: there should be an element in $\{(2,3), (2,4), (2,5), (2,6)\}$ since $(2,2)$ and $(4,2)$ share a coordinate.
\end{multicols}
\end{example}

\begin{example}
Take now $G$ to be the empty circles in the next figure, and $\C=(16,16)$. Then $[G]_\C$ is the set of marked dots.

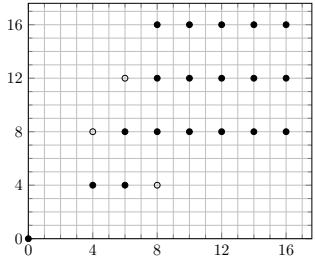
\begin{multicols}{2}
\begin{tikzpicture}[scale=.55]% coordinates
\begin{axis}[grid=both, minor tick num=3, xmin=0, ymin=0, xtick={0,4,...,16},  ytick={0,4,...,16}]
\addplot [only marks,mark=o] coordinates {
(4,8)
(8,4)
(6,12)

};
\addplot [only marks] coordinates {
(0,0)
(4,4)
(6,4)
(6,8)
(8,8)
(8,12)
(8,16)
(10,8)
(10,12)
(10,16)
(12,8)
(12,12)
(12,16)
(14,8)
(14,12)
(14,16)
(16,8)
(16,12)
(16,16)
};
\end{axis}
\end{tikzpicture}
\columnbreak

Again, $[G]_\C$ cannot be the set of small elements of a good semigroup, since (G2) would not be fulfilled.
\end{multicols}
\end{example}

Unfortunately, even if $G$ agrees with the conditions of Theorem \ref{th:carac-minimality}, the resulting monoid might not be good.
\begin{example}
Let $\C=(8,10)$ and $G=\{(3,4),(7,8)\}$. Then $[G]_C$ is
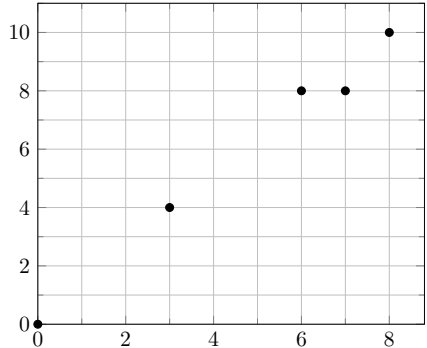
\begin{center}
\begin{tikzpicture}[scale=.75]
\begin{axis}[grid=both, minor tick num=1, xmin=0, ymin=0, xtick={0,2,..., 8}, ytick={0,2,..., 10}]
\addplot [only marks] coordinates {
(0,0)
(3,4)
(6,8)
(7,8)
(8,10)
};
\end{axis}
\end{tikzpicture}
\end{center}
\end{example}

\begin{remark}\label{C-1} 
It may also happen that we start with $G$ and $\C$, and
$[G]_\C$ contains $C-\mathbf e_i$ for some $i$. 
As a consequence of Lemma \ref{lem:contains-border-axes} the conductor would not be $\C$. 
In our implementation we allow this to happen, and we redefine $\C$.
\end{remark}

For semigroup duplications and amalgamations we can infer what is the minimal good generating system.

\begin{remark}
Let $S$ be a numerical semigroup minimally generated by $G$, and $E\subseteq S$ be an ideal generated by $A$. 
Let $\C(E)$ be the conductor of $E$. 
Recall that the semigroup duplication $S\bowtie E$ is a good semigroup. From the definition, it follows that $\mathrm \C(S\bowtie E)=(\C(E),\C(E))$. 
It is easy to prove that the set
\[(\{\C(E)\}\times A)\cup (A\times\{\C(E)\}) \cup \{(s,s)\mid s \in G\setminus (G\cap E)\}\]
%\[\left\lbrace (c(E),x_i),\ i\in\{1,\dots n\}\right\rbrace \cup \left\lbrace (x_i,c(E)),\ i\in\{1,\dots n\}\right\rbrace\cup\left\lbrace (s,s),\ s\ \mbox{generator of}\ S,\ s\notin E\right\rbrace\]
is a  minimal good system of generators for $S\bowtie E$.

Also, for $T$ a numerical semigroups, $g:S\to T$ a monoid morphism (and thus multiplication by an integer) and $E$ an ideal of $T$, we know that $S\bowtie^gE$ is also a good semigroup. 
Its conductor is $(\C(g^{-1}(E)),\C(E))$ with $\C(g^{-1}(E))$ and $\C(E)$ the conductors of $g^{-1}(E)$ and $E$, respectively. 
Again an easy check shows that the set
\[\left\lbrace (s,g(s))\mid s\in G\setminus (G\cap g^{-1}(E))\right\rbrace\cup
(B\times \{\C(E)\}) \cup (\{\C(g^{-1}(E))\}\times A)\]
%\left\lbrace (x,c(E)),\ x\ \mbox{generator of}\ g^{-1}(E)\subseteq S\right\rbrace\cup\left\lbrace (c(g^{-1}(E)),y),\ y\ \mbox{generator of}\ E\subseteq S\right\rbrace\]
is a good minimal system of generators for $S\bowtie^gE$, with $H$ a minimal generating system of $E$ and $B$ a minimal generating system of $g^{-1}(E)$.
\end{remark}

\subsection{Nonlocal case}

We already mentioned that every good semigroup is a direct product of good local semigroups (see \cite[Theorem 2.5]{a-u}).
In this case, the minimal good generating systems are not unique, as we can see in the following example. 

\begin{example}
Let $S$, $T$ and $S\times T$ as follows.
\begin{verbatim}
gap> s:=NumericalSemigroup(3,5,7);
<Numerical semigroup with 3 generators>
gap> t:=NumericalSemigroup(2,5);
<Modular numerical semigroup satisfying 5x mod 10 <= x >
gap> g:=cartesianProduct(s,t);
<Good semigroup>
gap> SmallElementsOfGoodSemigroup(g);
[ [ 0, 0 ], [ 0, 2 ], [ 0, 4 ], [ 3, 0 ], [ 3, 2 ],
  [ 3, 4 ], [ 5, 0 ], [ 5, 2 ], [ 5, 4 ] ]
\end{verbatim}
The elements $\{(0,4),(3,2),(5,0)\}$ and $\{(0,4),(3,4),(5,0), (5,2)\}$ are minimal good generating systems for $S\times T$. 
\end{example}

In general, if $S=\prod_i S_i$ is a non local semigroup (with $S_i$ local), it could be a natural choice to take as generating system the product $\prod_i A_i$, where $A_i$ is a minimal good generating system for $S_i$.
This is not a minimal good generating system but it is unique and reflects the fact that the semigroup is the Cartesian product of the projections. This choice is motivated also from the fact that the main properties of $\prod_i S_i$, like symmetry, can be read in terms of the same properties for $S_i$ (see \cite{a-u}).

%\textcolor{red}{If $x$ is in the border, then, as the examples show, in some cases it can be removed from the good generating system. Let us focus, without loss of generality, on $\partial_1(S)$. Then $y\in x+OY$. My guess is that if $y\not\in G$, then we can still remove $x$. But I have to check first if in the examples it was the case that $x=y$.}

%\todo[inline]{It is not true in general that we can find a minimal good generating system...? The example doesn't clearly say that this is the case.} You can always start with the set of small elements, which is finite. Then you remove an element that is in the semigroup spanned by the rest, and so on. So there is always a minimal good generating system. The point here is that it is not unique.

\section{Relative good ideals}\label{sec:ideals}

We now consider relative ideals of a good semigroup $S$.
We note that from (G1), $E$ has a minimal element $\mathrm m(E)$.
Being $E$ a relative ideal, we have $\mathrm
m(E)+(\C(S)+\mathbb N^n)\subseteq E$. Hence, as we mentioned in Section 2, we do not need
explicitly (G3) in the definition of good relative ideal, and relative ideals have a conductor. We denote the conductor of $E$ by $\C(E)$.

For $H$ a subset of $\mathbb N^n$, denote by $H+S$ the ideal of $S$ generated by $H$.

By Corollary \ref{smallelementsforideals}, a good
relative ideal is fully determined by its small elements.

Among the elements of $\mathrm{Small}(E)$, some might be of the
form $e+s$ with $s\in S\setminus\{(0,0)\}$ or infima of others,
and so in many cases we could choose a smallest subset of
$\mathrm{Small}(E)$ that still describes uniquely $E$.  To this
end, for $H\subseteq\mathbb N^n$, define $[H]$ to be the smallest
relative ideal of $S$ containing $H$ that is closed under infima.

Notice that with this definition we only consider relative ideals contained in $\mathbb N^n$. However this is not restrictive, since by definition of relative ideal there is always an $a$ such that $a+[H]\subset S\subseteq \mathbb N^n$. Moreover, if we consider $S$ to be local, we can also assume that the relative ideals so generated do not have any element on the axes.

Since intersections of ideals of $S$ closed under infima are again ideals of $S$ closed under infima, we get for relative ideals a result similar to Corollary \ref{cor:dist}.

\begin{proposition}\label{prop:distideal}
Let $H$ be a subset of $\mathbb N^n$. Then
\[
[H]=\{h_1\wedge\dots\wedge h_n\mid h_i\in(H+S)\}.
\]
\end{proposition}

We are mostly interested in when $[H]$ covers completely
$\mathrm{Small}(E)$ for some $H\subseteq \mathbb N^n$.
So we define $[H]_{\C(E)}=\C(E)\wedge [H]$.
We say that $H$ is a \emph{good generating system} for $E$ if
\[
[H]_{\C(E)}=\mathrm{Small}(E).
\]
We say that $H$ is \emph{minimal} if no proper subset of $H$ is a good generating system of $E$.

Observe that $[H]_{\C(E)}=[\C(E)\wedge H]_{\C(E)}$, thus we will always assume that good systems of generators are
contained in $\mathrm{Small}(E)$. Clearly $[\mathrm{Small}(E)]_{\C(E)}=\mathrm{Small}(E)$, so that $\mathrm{Small}(E)$ is always a good generating system. However, it doesn't need to be minimal. We would like, when possible, to remove redundant elements of $\mathrm{Small}(E)$.

Let us assume $S$ is local. By conditions (G1) and (G2), using
similar arguments as in Section \ref{sec:good-gen-sys}, we get the following results.

\begin{lemma}\label{membership-good-system-ideals}
Let $d\in\mathbb N^n$ and $G\subseteq \mathrm B(d)$. Let $a\in \mathrm B(d)\setminus\{d\}$ and let $J$ be maximal with the property $d\in \Delta_J(a)$. Then $a\in [G]_d$ if and only if $\bar\Delta_i(a)\cap(G+S)\not=\emptyset$ for all $i\in I\setminus J$.
\end{lemma}

The following lemma differs slightly from Lemma \ref{lem:shared-coord}, and for this reason we include a proof.

\begin{lemma}\label{lem:shared-coord-ideal}
Let $S$ be a good local semigroup and $E$ a good relative ideal. Let $H$ be a good generating system for $E$ and $a\in H\setminus\partial(E)$. If there exists $b\in \bar\Delta(a)\cap((H\setminus\{a\})+S)$, then $H\setminus\{a\}$ is a good generating system for $E$.
\end{lemma}
\begin{proof}
Since $a\not\in\partial(E)$ we have $a<\C(E)$. 
If $a=b<\C(E)$, then we are done, since then clearly $a\in[H\setminus\{a\}+S]_{\C(E)}$. 
So let us suppose $a\not=b$. 
By assumption we have $H\subseteq \mathrm{Small}(E)$ and $[H]_{\C(E)}=\mathrm{Small}(E)$. 
Since $b\in \bar\Delta(a)$, there exists an $i\in I$ such that $b_i=a_i$. 
Then, applying (G2), one can find  a $d\in E\setminus\{a,b\}$ such that 
\[
a=b\wedge d. 
\]

\noindent Eventually substituting $d$ with $d\wedge \C(E)$ we can assume $d\in \mathrm{Small}(E)=[H]_{\C(E)}$. 
Since $a\not\in\partial(E)$, after this substitution we still have $d\not=a\not=b$. 

\noindent By Proposition \ref{prop:distideal}, we can write
\[
d=\min\{\C(E),d^{(i)}\mid i\in I\}, \text{ with }\{d^{(i)}\}_{i\in I} \subset (H+S).
\]
As $d\not=a$ and $a<d\le d^{(i)}$, we also have $a\not=d^{(i)}$ for any $i\in I$.
Let $J\subseteq I$ be the maximal set of indices such that $d\in\bar\Delta_J(a)$ (which implies $b\in\bar\Delta_{I\setminus J}(a)$). 
For every $j\in J$ there exists an $i_j\in I$ such that $d^{(i_j)}_j=d_j=a_j$.
Hence for any $j\in J$, $d^{(i_j)}\in\bar\Delta_j(a)$ and moreover 
\[
a=\min\{b,d^{(i_j)}\mid j\in J\}.
\] 
Since $d^{(i_j)}\in (H+S)$, for any $j\in J$ we can write
\[
d^{(i_j)}=g^{(i_j)}+s^{(i_j)}
\]
where $g^{(i_j)}\in H$ and $s^{(i_j)}\in S$ for any $j\in J$. This yields
\[
d^{(i_j)}=g^{(i_j)}+s^{(i_j)}\not=a
\]
and
\[
d^{(i_j)}_j=a_j.
\]
This implies one of the following three possibilities: either
$s^{(i_j)}_j=a_j$ and $g^{(i_j)}$ belongs to the axes, which is not possible since we assume $E\subseteq S$ and $S$ local; or
$g^{(i_j)}_j,s^{(i_j)}_j<a_j$ and  $d^{(i_j)}\in (H\setminus\{a\}+S)$ or, as $S$ is local,
$s^{(i_j)}=0$ and $d^{(i_j)}=g^{(i_j)}_j\in H\setminus\{a\}$.

\noindent Therefore $d^{(i_j)}\in(H\setminus\{a\}+S)$ for any $j\in J$ and $a=\min\{b,d^{(i_j)}\mid j\in J\}\in[H\setminus\{a\}]_{\C(E)}$.
\end{proof}

The following is the analogous of Theorem \ref{th:carac-minimality}. We omit the proof since it is completely analogous to the one done for the semigroup. 

\begin{theorem}\label{th:carac-minimalityideal}
Let $S$ be a good local semigroup, $E$ a good semigroup of $S$ and $H$ a good positive generating system for $E$. For $a\in H$, let $J_a$ be the set of indices maximal with the property $a\in\partial_{J_a}(E)$ ($J_a$ can be empty). Then $H$ is a minimal good generating system if and only if 
\[
\bar\Delta(a)\cap ((H\setminus\{a\})+S)=\emptyset, \text{ if } J_a=\emptyset,
\]
and there exists an $i\in I\setminus J_a$ such that
\[
\bar\Delta_i(a)\cap ((H\setminus\{a\})+S)=\emptyset, \text{ if } J_a\ne\emptyset.
\]
\end{theorem}

As for the good local semigroup case, we have that minimal good generating systems are unique. 
The proof is similar to the one of Theorem \ref{th:unique}, so we do not include it.

\begin{theorem}\label{th:uniqueideal}
Let $S$ be a good local semigroup and $E$ a good semigroup ideal. Then $E$ has a unique minimal good generating system.
\end{theorem}

\subsection{The canonical ideal}

There is a distinguished good relative ideal that plays an
important role in many properties of good semigroups: the
canonical ideal. Before giving its definition we recall the notion
of symmetry for good semigroups.

Let $\C$ be the conductor of $S$ and let
$\gamma=\C-\textbf{1}$, as defined in Section \ref{sec:good-sem}. 

In \cite{delgado}, a semigroup $S$ is said to be \emph{symmetric} if
\[
s\in S\ \hbox{if and only if }\mathrm \Delta^S(\gamma -s)=\emptyset.
\]

The subset $\K=\K(S)=\{a\in\mathbb Z^n\mid \Delta^S(\gamma-a)=\emptyset\}$ is a relative ideal (see
\cite[Proposition 2.14]{a-u}) and it is called the \emph{canonical
ideal} of $S$. A more general definition has been recently given
in \cite{K-T-S}. Hence $S$ is symmetric if and only if $\K=S$. In
\cite{danna1} it is proved that $\K$ satisfies (G1) and (G2), hence
it is a good relative ideal of $S$.

\begin{lemma}
Let $S$ be a local good semigroup. Then
\begin{itemize}
\item $\K\subseteq \mathbb N^n$,
\item $\C(\K)=\C$.
\end{itemize}
\end{lemma}
\begin{proof}
Let $a\in \mathbb Z^n$. Assume that $a_i< 0$. Then $(\gamma-a)_i\ge\C_i$, which means that $\Delta^S(\gamma-a)$ is not empty. Hence $a\not\in \K$. Hence $\K\subseteq \mathbb N^n$.

Now take $a\ge \C$. Then $\gamma-a<0$, and thus $\Delta^S(\gamma-a)=\emptyset$. This proves that $\C+\mathbb N^n\subseteq \K$.

Consider now $\gamma+\textbf{e}_i$, where $\textbf{e}_i$ is the $i$-th element of the canonical base of $\mathbb N^n$. We have that $\textbf{0}\in \Delta^S(\gamma-(\gamma+\textbf{e}_i))=\Delta^S(-\textbf{e}_i)$ for all $i\in I$. And so $\gamma+\textbf{e}_i\not\in \K$ for all $i\in I$. This implies that $\C=\C(\K)$.
\end{proof}

Hence a possible way to compute $\K$ is to determine which
elements $a$ such that $\textbf{0}\le a\le \C$ satisfy the condition $\Delta^S(\gamma-a)=\emptyset$.

%Our next aim is to find a good generating system for $\K$, and in order to get this we need three lemmas. We define $S_i:=\pi_i(S)$ for $i\in I$ where $\pi_i$ is the natural projection. We will denote $S_J=\pi_J(S)$ the projection with respect to a set of indices $J\subseteq I$
%
%\begin{lemma}\label{beta1}
%Given an element $b$ such that $C\in\overline\Delta_J(b)$ for some maximal $\emptyset\not=J \subset I$, $b\in \K$ if and only if $\Delta^S_i(\gamma-b)=\emptyset$ for all $i\in I\setminus J$.
%\end{lemma}
%\begin{proof}
%Since $\C\in\overline\Delta_J(b)$, we have $(\gamma-b)_j=-1$ for all $j\in J$.
%Therefore $\Delta^S_j(\gamma-b)=\emptyset$ for all $j\in J$.
%By definition of $K$ we have $b\in \K$ if and only if $\Delta^S(\gamma-b)=\emptyset$. 
%Since $\Delta^S_j(\gamma-b)=\emptyset$ for all $j\in J$, this is equivalent to $\Delta^S_i(\gamma-b)=\emptyset$ for all $i\in I\setminus J$.
%\end{proof}
%
%We say that an element $b$ is \emph{maximal} in a good relative ideal $E$, when
%$\Delta^E(b)= \emptyset$. 
%The notion of maximality
%comes from the fact that $b$ will be maximal in its
%$i$-th fiber (the elements in $E$ having the same $i$-th
%coordinate as $b$). 
%In 2 dimensions: Note that from
%the definition of good relative ideal, being maximal in the
%vertical fiber is equivalent to being maximal in the horizontal
%fiber. 
%The equivalent in general doesn't exist cause property (G2) controls only two entries at a time.
%The next lemma is proved in in the two branches case in \cite{two}, %with the additional hypotesis $b$ maximal in $S$
%but we give here a proof of the general case.

In order to compute more efficiently $K$, it would be important to find a generating system for it.

In the following we give a good generating system in the two dimensional case, while the general case remains open. So from now till the end on the section it will be $I=\{1,2\}$.

We define $S_i:=\pi_i(S)$ for $i\in I$ where $\pi_i$ is the natural projection.

\begin{lemma}\label{beta1}
An element $b$, with $b_i=C_i$ for some $i\in I$, is in $\K$ if and only if $\gamma_j-b_j\notin S_j$, where $\{j\}=I\setminus\{i\}$.
\end{lemma}
\begin{proof} 
It is not restrictive to assume $i=2$.
By definition of $K$ we have that $b\in \K$ if and only if $\Delta^S(\gamma-b)=\emptyset$. Since $\gamma-b=(\gamma_1-b_1,-1)$, we have $\Delta^S_2(\gamma-b)=\emptyset$. Hence $\Delta^S(\gamma-b)=\emptyset$ if and only if $\Delta^S_1(\gamma-b)=\emptyset$, that is $\gamma_1-b_1\notin S_1$.
\end{proof}

%\begin{lemma}\label{beta1}
%An element $b=(b_1,C_2)$ is in $K$ if and only if $\gamma_1-b_1\notin S_1$.
%\end{lemma}
%\begin{proof} By definition of $K$ we have that $b\in K$ if and only if $\Delta^S(\gamma-b)=\emptyset$. Since $\gamma-b=(\gamma_1-b_1,-1)$, we have $\Delta^S_2(\gamma-b)=\emptyset$. Hence $\Delta^S(\gamma-b)=\emptyset$ if and only if $\Delta^S_1(\gamma-b)=\emptyset$, that is $\gamma_1-b_1\notin S_1$.
%\end{proof}
%
%Reversing the indices and using the same proof as in the last lemma we get the following lemma.
%
%\begin{lemma}\label{beta2}
%An element $b=(C_1,b_2)$ is in $K$ if and only if $\gamma_2-b_2\notin S_2$.
%\end{lemma}

As we already said, an element $b$ is maximal in $\K$ when $\Delta^\K(b)= \emptyset$. 
In the two dimensional case, the notion of maximality is very natural. 
In fact, it comes from the fact that $b$ will be maximal in both its vertical fiber (the elements in $\K$ having the same second
coordinate as $b$) and its horizontal fiber. 
Note that from the definition of good relative ideal, being maximal in the vertical fiber is equivalent to being maximal in the horizontal fiber. 
The next lemma it is proved in \cite{two}, but for the convenience of the reader we give a proof also here.

\begin{lemma}\label{beta}
	Let $b=(b_1,b_2)\in K$ with $b\le\gamma$. Then  $b$ is maximal in $\K$ if and only if $\gamma-b\in S$ and $\Delta^S(\gamma-b)=\emptyset$.
\end{lemma}
\begin{proof} 
Let $b\in K$. 
Then $\Delta^S(\gamma-b)=\emptyset$. 
Let us suppose that $\gamma-b\notin S$.
 If $\gamma_1-b_1\notin S_1$, then  $\Delta^S((\gamma_1-b_1,-1))=\emptyset$. 
 Hence $(b_1,C_2)\in K$ and $b$ is not maximal  in $K$. The same applies to $\gamma_2-b_2$. 
 Thus $\gamma_1-b_1\in S_1$ and $\gamma_2-b_2\in S_2$. Let $y=(\gamma_1-b_1,y_2)\in S$. 
 Since $\Delta^S(\gamma-b)=\emptyset$, we have $y_2<\gamma_2-b_2$. 
 Assume that $\Delta^S(y)=\emptyset$. 
 We get $\gamma-y=(b_1,\gamma_2-y_2)\in K$ and $b$ is not maximal in $K$ as $\gamma_2-y_2>b_2$, which is absurd. 
 Hence for every $y\in S$ such that $y=(\gamma_1-b_1,y_2)$ for some $y_2\in \mathbb N$, we have that $y_2<\gamma_2-b_2$ and $\Delta^S(y)\neq \emptyset$. 
 If we choose $y_2$ to be maximum, then $\Delta^S(y)\neq \emptyset$ implies that there exists $y'=(y_1',y_2')\in S$ such that $y_1'>y_1$ and $y_2'=y_2$. But then condition (G2) forces the existence of $y''=(y_1'',y_2'')\in S$ with $y_1''=y_1=\gamma_1-b_1$ and $y_2''> y_2$. But this contradicts the maximality of $y_2$. This shows that $\gamma-b\in S$.
	
Let now suppose that  $\gamma-b\in S$ and $\Delta^S(\gamma-b)=\emptyset$. 
The second condition implies, by definition, that $b \in K$. 
If it is not maximal, there exist either $(b_1,y_2) \in K$, with $y_2>b_2$, or $(y_1,b_2)\in K$, with $y_1>b_1$. 
But $\gamma-b\in \Delta^S_2(\gamma_1-b_1,\gamma_2-y_2)$ and $\gamma-b\in \Delta^S_1(\gamma_1-y_1,\gamma_2-b_2)$ a contradiction against $(b_1,y_2) \in K$ and $(y_1,b_2)\in K$. Hence $b$ is maximal in $K$.
\end{proof}

\begin{remark}
Observe that Lemma \ref{beta} does not generalize to the $n$ dimensional case.
\end{remark}
Now we are ready to give a tentative good generating system for the canonical ideal.

\begin{proposition}\label{gens-canonical}
	A good generating system of generators for $\K$ is given by the following elements:
	\begin{itemize}
		\item $(\gamma_1-x_1,C_2)$ for $x_1\not\in S_1$,
		\item $(C_1,\gamma_2-x_2)$ for $x_2\notin S_2$,
		\item $\gamma-\alpha$ for $\alpha\in S$ with $\Delta^S(\alpha)=\emptyset$.
	\end{itemize}
\end{proposition}
\begin{proof} The proof follows immediately by the previous three lemmas.
\end{proof}

In general the good generating system of Proposition \ref{gens-canonical} is not minimal.

\begin{example}
Let us calculate the canonical ideal of $S\bowtie E$, with
$S=\langle 3,5\rangle$ and $E=0+S$, which in this case is the
canonical ideal of the numerical semigroup $S$. In this case,
Proposition \ref{gens-canonical} yields the following:
\begin{verbatim}
gap> s:=NumericalSemigroup(3,5);;
gap> s:=NumericalSemigroup(3,5);;
gap> e:=3+s;
<Ideal of numerical semigroup>
gap> c:=canonicalIdealOfGoodSemigroup(g);
[ [ 0, 0 ], [ 3, 11 ], [ 5, 5 ], [ 6, 11 ], [ 8, 11 ], [ 9, 11 ],
  [ 10, 10 ], [ 11, 3 ], [ 11, 6 ], [ 11, 8 ], [ 11, 9 ] ]
\end{verbatim}
And in this case $\{(0,0)\}$ is a generating system for the canonical ideal of $S$.
\end{example}

%\begin{remark}
%There are some natural operations between relative ideals $E$ and $F$, for instance, $E\cap F$, $E+F$ and $E-F=\left\lbrace x\in\mathbb Z^2\ |\ x+F\subseteq E\right\rbrace $. Unfortunately when the arguments of these operations are good ideals, the resulting ideal might not be good. It is not difficult to find examples for these facts.
%\end{remark}

\section{Arf good semigroups}\label{sec:arf}

Let $S\subseteq \mathbb N^n$ be a semigroup and $E$ be a relative ideal of $S$, with $e=\mathrm m(E)$. Then $E$ is called \emph{stable} if $E+E=e+E$. It is well known that the stability of $E$ is equivalent to $E-E=E-e$.
The semigroup $S$ is \emph{Arf} if for all $a\in S$, $S(a)$ is stable, where
\[S(a)=\{ b\in S \mid b\ge a\}.\]
We notice that the inclusion $S(a)-S(a)\subseteq S(a)-a$ is always
true, since if $z\in S(a)-S(a)$, $z+b\in S(a)$ for all $b\in S(a)$
and, in particular, $z+a\in S(a)$.

In the numerical semigroup case, there are many alternative definitions and characterizations (see for instance \cite{bdf}).
We show next that the following characterization still holds for
subsemigroups of $\mathbb N^n$.

\begin{lemma}\label{translate-Arf}
A subsemigroup $S$ of $\mathbb N^n$ has the Arf property if and only if for all $a,b,c\in S$, with $b\ge a$ and $c\ge a$, we have $b+c-a\in S$.
\end{lemma}
\begin{proof}
\emph{Necessity.} Assume that $a,b,c\in S$ with $b\ge a$ and $c\ge a$. As $b-a\in S(a)-a$ and, by hypothesis, $S(a)-a=S(a)-S(a)$, we have that $(b-a)+S(a)\subseteq S(a)$. Since $c\in S(a)$, we get $b-a+c=b+c-a\in S(a)\subseteq S$.

\emph{Sufficiency.} We know that $S(a)-S(a) \subseteq S(a)-a$. For the other inclusion, assume now that $z\in S(a)-a$ and we must show that $z\in S(a)-S(a)$.
As $z\in S(a)-a$, we have that $z=b-a$ for some $b\in S(a)$. Take $c\in S(a)$.
Then $z+c=b+c-a$, which is in $S(a)$ by hypothesis (it is in $S$ and $b+c-a=b+(c-a)\ge a$).
Hence $z+S(a)\subseteq S(a)$ and thus $z\in S(a)-S(a)$.
\end{proof}

One could ask if it true that, in order to verify if a good semigroup has the Arf property it
suffices to check the above condition in the set of small elements of the good semigroup. We can only prove it for the case $n=2$.

\begin{proposition}\label{effective-Arf}
Let $S\subseteq \mathbb N^2$ be a good semigroup. 
Then $S$ has the Arf property if and only if for all $a,b,c\in \mathrm{Small}(S)$, with $b\ge c$ and $b\ge  a$, we have $b+c-a\in S$.
\end{proposition}
\begin{proof}
Clearly, if $S$ has the Arf property, by
Lemma \ref{translate-Arf}, we have that for all $a,b,c\in
\mathrm{Small}(S)$, with $b\ge a$ and $c\ge a$, we have $b+c-a\in
S$.

For the converse, let $a,b,c\in S$ with $b\ge a$ and $c\ge a$. We have to prove that $b+c-a\in S$. Notice that $b+(c-a)\ge b$ and $b+c-a=c+(b-a)\ge c$. This  if either if $b$ or $c$ are greater than $C$, then $b+(c-a)\in S$. Also if $b,c\in \mathrm{Small}(S)$, then $a\in \mathrm{Small}(S)$ too and by hypothesis $b+c-a\in S$.

So it remains to see what happens when $b$ or $c$ are in the upper or right bands of $\mathrm{Small}(S)$.

Assume that $b_1< C_1$ and $b_2\ge C_2$. Then $a_1<C_1$. Then $b_2+c_2-a_2\ge C_2$. Clearly if $b_1+c_1-a_1\ge C_1$, then $b+c-a\ge C$ and consequently $b+c-a\in S$. If $b_1+c_1-a_1<C_1$, then $b+c-a\in S$ if and only if $(C_2,b_1+c_1-a_1)\in S$ by Proposition \ref{prop:membership}. Take $b'=b\wedge C$, $c'=c\wedge C$ and $a'=a\wedge C$. Then $b'+c'-a'=(C_2,b_1+c_1-a_1)$, and $a',b',c'\in \mathrm{Small}(S)$. By hypothesis $b'+c'-a' \in S$, and this leads to $b+c-a\in S$.

The case $b_2<C_2$ and $b_2\ge C_1$ is similar. 
%!!!! NOT PROVED YET Since $b_j< \C_j$ for all $j\in J$, then $c_j< \C_j$ and $a_j< \C_j$ for all $j\in J$. 
%Let $k\subseteq J$ be maximum such that $b_k+c_k-a_k<\C_k$ for $k\in K$, i.e. $b_j+c_j-a_j\ge \C_j$ for $j\in J\setminus K$.
%Then $c_k-a_k<\C_k-b_k<\C_k$ for $k\in K$ and
%$c_j-a_j>\C_j-b_j<\C_j$.
%Take $b'=b\wedge \C$, $c'=c\wedge \C$ and $a'=a\wedge \C$. 
%We have $((b+c-a)\wedge \C)_k=(b\wedge \C)_k+((c-a)\wedge \C)_k=b_k+c_k-a_k=(b\wedge \C)_k+(c\wedge \C)_k-(a\wedge \C)_k$ for $k\in K$.
%For $i\in I\setminus J$, $((b+c-a)\wedge \C)_i=\C_i$.
%For $j\in J\setminus K$, $((b+c-a)\wedge \C)_j=\C_j$ but $(b\wedge \C)_j+(c\wedge \C)_j-(a\wedge \C)_j=b_j+c_j-a_j$.
%By Proposition \ref{prop:membership}, $b+c-a\in S$ if and only if $(b+c-a)\wedge \C\in S$.
%From above $b'+c'-a'=(b+c-a)\wedge \C$, and $a',b',c'\in \mathrm{Small}(S)$. By hypothesis $b'+c'-a' \in S$, and this leads to $b+c-a\in S$.
\end{proof}

\begin{remark}\label{stableandarf}
In  \cite[Lemma 3.20]{a-u}, it is proved that, for a good semigroup $S$, the ideal $S(a)$ is stable for every $a\le \C(S)$ if and only if $S$ has the Arf property. So, if one wants to check the Arf property for a good semigroup $S$, it is more convenient to work with ideals rather than elements. 
\end{remark}

We would like to define the Arf closure of a good semigroup $S$ as the smallest Arf good semigroup containing $S$;
to do this we need to prove that the intersection of good Arf semigroups containing $S$ is a good Arf semigroup.

\begin{lemma}\label{stable}
$E$ is stable if and only if $E-e$ is a semigroup.
In particular, $E$ is a stable good ideal if and only if $E-e$ is a good semigroup.
\end{lemma}
\begin{proof}
If $E$ is stable, then $E-e=E-E$ is a semigroup.

As for the converse, let $E-e$ be a semigroup and let us prove that $E+E\subseteq e+E$ (the other inclusion is trivial).
Let $b,c\in E$. By hypothesis $(b-e)+(c-e)=d-e$ for some $d\in E$. Hence $b+c=e+d\in e+E$.

The second statement is straightforward.
\end{proof}

\begin{proposition}\label{Arfclosure}
Let $S$ be a good semigroup. The intersection of good Arf semigroups containing $S$ is a good Arf semigroup.
\end{proposition}
\begin{proof}
Let us prove that the intersection of two good Arf semigroups containing $S$, say $T$ and $U$, is a good Arf semigroup. 
This is enough as the number of semigroups in $\mathbb N^n$ containing $S$ is finite.

It is straightforward to check that the intersection of two
Arf semigroups (not necessarily good) is an Arf semigroup. 
Hence we need to check that the intersection of two Arf good semigroup $T$ and $U$ is good. 

Conditions (G1) and (G3) trivially hold for $T\cap U$. So we need only to check condition (G2).
Let $a\in T\cap U$.
Being $T\cap U$ an Arf semigroup as said above, we have that $(T\cap U)(a)$ is stable by definition.
Therefore, by Lemma \ref{stable}, being $a=\mathrm{m}(T\cap U(a))$, we get $(T\cap U)(a)-a$ is a good semigroup. 
In particular, condition (G2) holds for $(T\cap U)(a)-a$. 
This implies that it also holds in $(T\cap U)(a)$, and we are done.
\end{proof}

\begin{corollary}\label{Arf-cl}
Let $S$ be a good semigroup. Then there exists an Arf good semigroup $T$ containing $S$, such that for every Arf good semigroup $U$ containing $S$, the inclusion $T\subseteq U$ holds.
\end{corollary}

We call the semigroup $T$ in the previous corollary the \emph{Arf closure} of $S$.

%\begin{lemma}\label{Salpha}
%A semigroup $S$ is Arf if and only if $S(a)-a$ is Arf for every $a\in S$.
%\end{lemma}
%\begin{proof}
 %   It follows by $(b-a)+(\gamma-a)-(c-a)=(b+\gamma-c)-a$.
%\end{proof}

%\begin{proposition}
 %   Let $S$ be a local semigroup and let $e=\min  \left\lbrace S\setminus\left\lbrace (0,0)\right\rbrace\right\rbrace $.
 %   Then $S$ is Arf if and only if $S(e)-e$ is an Arf semigroup.
%\end{proposition}
%\begin{proof}
 %   The necessary implication is trivial by Lemma \ref{Salpha}.
  %  Let $b\in S$, hence $b\ge e$. Let $T=S(e)-e$.
   % Then $T$ is Arf by hypothesis and since $S(b)-b=T(b-e)-(b-e)$, we get the thesis.
%\end{proof}

\begin{remark}
Let $S_i, i\in I$ be the projections on the $i-th$ axes of $S$. Let $T_i$ be the Arf closure of $S_i$ for all $i\in I$. Then $\prod_i T_i$ is an Arf good semigroup containing $S$. Hence the Arf closure $T$ of $S$ is contained in $\prod_i T_i$.
\end{remark}

\begin{lemma}\label{projection}
Let $U$ be any Arf good semigroup such that  $S\subseteq U \subseteq \prod_iT_i$. Then $\pi_i(U)=T_i$ for all $i\in I$.
\end{lemma}
\begin{proof}
As $U\subseteq \prod_i T_i$, we have that $\pi_i(U)\subseteq T_i$. 
So we only need to prove that $T_i\subseteq\pi_i(U)$. 
As $S\subseteq U$, we have $S_i\subseteq \pi_i(U)$. 
Thus if we show that $\pi_i(U)$ is an Arf semigroup, then $T_i\subseteq \pi_i(U)$, and we are done.

Let $i\in I$.
Recall that $\pi_i(U)$ is an Arf numerical semigroup if and only if for any $a,b,c\in \pi_i(U)$ with $a\ge c$ and $b\ge c$ we have $a+b-c\in \pi_i(U)$. 
So take $a,b,c\in \pi_i(U)$ with $a\ge c$ and $b\ge c$. Let $a',b',c'\in U$ be such that $a=\pi_i(a'), b=\pi_i(b'), c=\pi_i(c')$. 
If $a'\ge c'$ and $b'\ge c'$, as $U$ has the Arf property, $a'+b'-c'\in U$, and consequently $a+b-c\in \pi_i(U)$. 
If $a'\not\ge c'$, then we substitute $c'$ with $a'\wedge c'\in U$ ($U$ is good). 
We then still have $\pi_i(a'\wedge c')=\pi_i(a')\wedge\pi_i(c')=a\wedge c=c$ and $a'\ge a'\wedge c'$.
If now $b'\not\ge a'\wedge c'$, again we substitute $a'\wedge c'$ with $(a'\wedge c')\wedge b'$, which is also in $U$.
We then still have $\pi_i((a'\wedge c')\wedge b')=(\pi_i(a')\wedge\pi_i(c'))\wedge\pi_i(b')=c\wedge b=c$ and $b'\ge (a'\wedge c')\wedge b'$.
So now renominating $c'=(a'\wedge c')\wedge b'$ we have $a'+b'-c'\in U$ because $U$ is Arf. 
But then $\pi_i(a'+b'-c')=a+b-c \in \pi_i(U)$.
\end{proof}

In the two dimensional case it is easy to use the previous lemma in order to give an algorithm to compute the Arf closure. We explain in detail such an algorithm in the rest of the section.
In the general case the problem is more involved.
In \cite{zito} it is given a procedure to find all the Arf semigroups with a given multiplicity tree and, consequently, the Arf closure of a given semigroup.

In the two dimensional case we only have two projections from Lemma \ref{projection}, and we call them $T_1$ and $T_2$.

Both $T_1$ and $T_2$ have a multiplicity sequence, say
$\left\lbrace e_i\right\rbrace_{i\ge 0} $ for $T_1$ and
$\left\lbrace f_i\right\rbrace_{i\ge 0} $ for $T_2$, respectively,
which characterize them. With these multiplicity sequences it is
possible to construct the set of all good Arf
semigroups with $T_1$ and $T_2$ as their projections (see
\cite{arf}). In our case, $S\subseteq\mathbb{N}^2$, the set of
good Arf semigroups with $T_1$ and $T_2$ as
projections is totally ordered by inclusion (see Lemma 5.1,
\cite{arf}), following the ordering established by the
multiplicity trees (see Figure \ref{fig:trees}).

%ADD THE FIGURE

By definition of Arf closure and by Lemma \ref{projection}, the Arf closure $T$ of $S$ is
the smallest of such semigroups containing $S$ and it is obtained
as finite sums of the multiplicity vectors taken in a subtree
rooted in $(e_0,f_0)$.

If $T_1=\left\lbrace 0,s_1,s_2,\dots\right\rbrace $ and $T_2=\left\lbrace 0,u_1,u_2,\dots\right\rbrace $, let
$$
T^{(i)}=\left\lbrace (0,0),(s_1,u_1),\dots,(s_{i-1},u_{i-1})\right\rbrace \cup \left\lbrace T_1(s_i)\times T_2(u_i)\right\rbrace
$$
for every $i\ge 1$. The procedure in order to find the
Arf closure of a good semigroup $S$ is the following.

\begin{itemize}
   \item Calculate $S_i=\pi_i(S)$ for $i\in\{1,2\}$.
   \item Calculate the Arf closure $T_i$ for $i\in \{1,2\}$ (as explained for instance in \cite{arf-num}).
   \item Compute $T^{(2)}$.
   \item If $S\not\subseteq T^{(2)}$, then $T=T^{(1)}$. If $S\subseteq T^{(2)}$, then calculate $T^{(3)}$.
   %\item If $S\not\subseteq T^{(3)}$, then $T=T^{(2)}$. If $S\subseteq T^{(3)}$, then calculate $T^{(4)}$.
   \item  Repeat this process  until $S\subseteq T^{(i)}$ and $S\not\subseteq T^{(i+1)}$. Then $T=T^{(i)}$.
\end{itemize}

%\begin{corollary}
%Let $\left\lbrace e=a^{(1)}<\cdots<a^{(n)}=C(S)\right\rbrace $ a saturated chain in $S$. Then $S$ is Arf if and only if $S(a^{(i)})-a^{(i)}$ is an Arf semigroup for every $i=1,\dots,n$.
%\end{corollary}
%\begin{proof}
%   The necessary implication follows by Lemma \ref{Salpha}. The sufficient implication follows by last proposition with $a^{(i)}=e$.
%   \end{proof}

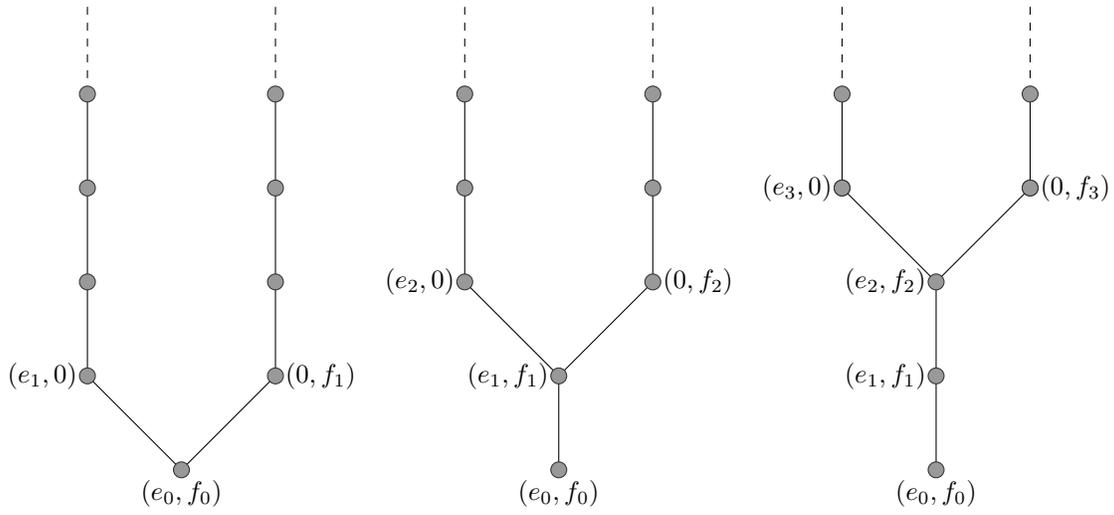
\begin{figure}[h]
\begin{tikzpicture}[y=.25cm, x=.25cm,font=\small]

\draw[dashed] (10,3) -- (10,8);

\draw[dashed] (0,3) -- (0,8);

\draw (10,-12) -- (10,3);

\draw (0,-12) -- (0,3);

\draw (5,-17) -- (0,-12);

\draw (5,-17) -- (10,-12);

\filldraw[fill=black!40,draw=black!80] (5,-17) circle (3pt)    node[anchor=north] {$(e_0,f_0)$};

\filldraw[fill=black!40,draw=black!80] (0,-12) circle (3pt)    node[anchor=east] {$(e_1,0)$};

\filldraw[fill=black!40,draw=black!80] (10,-12) circle (3pt)    node[anchor=west] {$(0,f_1)$};

\filldraw[fill=black!40,draw=black!80] (0,-7) circle (3pt)   ;

\filldraw[fill=black!40,draw=black!80] (10,-7) circle (3pt)  ;

\filldraw[fill=black!40,draw=black!80] (0,-2) circle (3pt)   ;

\filldraw[fill=black!40,draw=black!80] (10,-2) circle (3pt)  ;

\filldraw[fill=black!40,draw=black!80] (0,3) circle (3pt)   ;

\filldraw[fill=black!40,draw=black!80] (10,3) circle (3pt)  ;

\end{tikzpicture}
\begin{tikzpicture}[y=.25cm, x=.25cm,font=\small]

\draw[dashed] (10,3) -- (10,8);

\draw[dashed] (0,3) -- (0,8);

\draw (10,-7) -- (10,3);

\draw (0,-7) -- (0,3);

\draw (5,-12) -- (0,-7);

\draw (5,-12) -- (10,-7);

\draw (5,-17) -- (5, -12);

\filldraw[fill=black!40,draw=black!80] (5,-17) circle (3pt)    node[anchor=north] {$(e_0,f_0)$};

\filldraw[fill=black!40,draw=black!80] (5,-12) circle (3pt)    node[anchor=east] {$(e_1,f_1)$};

\filldraw[fill=black!40,draw=black!80] (0,-7) circle (3pt)   node[anchor=east] {$(e_2,0)$};

\filldraw[fill=black!40,draw=black!80] (10,-7) circle (3pt)  node[anchor=west] {$(0,f_2)$};

\filldraw[fill=black!40,draw=black!80] (0,-2) circle (3pt)   ;

\filldraw[fill=black!40,draw=black!80] (10,-2) circle (3pt)  ;

\filldraw[fill=black!40,draw=black!80] (0,3) circle (3pt)   ;

\filldraw[fill=black!40,draw=black!80] (10,3) circle (3pt)  ;

\end{tikzpicture}
\begin{tikzpicture}[y=.25cm, x=.25cm,font=\small]

\draw[dashed] (10,3) -- (10,8);

\draw[dashed] (0,3) -- (0,8);

\draw (10,-2) -- (10,3);

\draw (0,-2) -- (0,3);

\draw (5,-7) -- (10,-2);

\draw (5,-7) -- (0,-2);

\draw (5,-17) -- (5, -7);

\filldraw[fill=black!40,draw=black!80] (5,-17) circle (3pt)    node[anchor=north] {$(e_0,f_0)$};

\filldraw[fill=black!40,draw=black!80] (5,-12) circle (3pt)    node[anchor=east] {$(e_1,f_1)$};

\filldraw[fill=black!40,draw=black!80] (5,-7) circle (3pt)   node[anchor=east] {$(e_2,f_2)$};

\filldraw[fill=black!40,draw=black!80] (10,-2) circle (3pt)  node[anchor=west] {$(0,f_3)$};

\filldraw[fill=black!40,draw=black!80] (0,-2) circle (3pt)  node[anchor=east] {$(e_3,0)$};

\filldraw[fill=black!40,draw=black!80] (0,-2) circle (3pt)   ;

\filldraw[fill=black!40,draw=black!80] (10,-2) circle (3pt)  ;

\filldraw[fill=black!40,draw=black!80] (0,3) circle (3pt)   ;

\filldraw[fill=black!40,draw=black!80] (10,3) circle (3pt)  ;

\end{tikzpicture}
\caption{Multiplicity trees}
\label{fig:trees}
\end{figure}

\begin{example}\label{arfex}
Let $S$ be the good semigroup $[(4,3),(3,4)]_{(6,7)}$. We have on the left a picture of it and on the right a picture of its Arf closure.

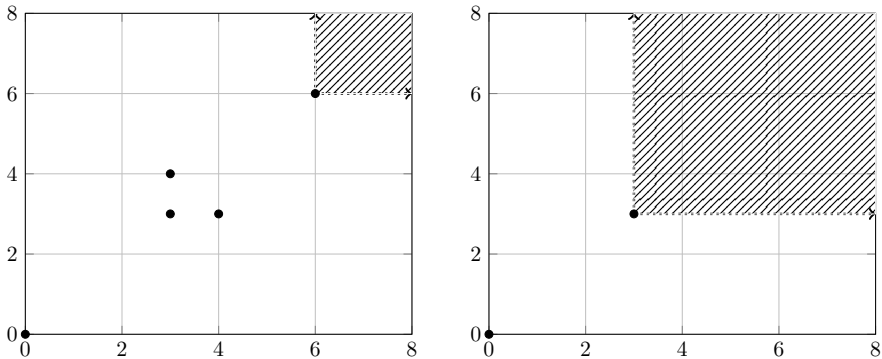
\begin{center}
\begin{tikzpicture}[scale=.75]
 \begin{axis}[grid=both, xmin=0,xmax=8, ymin=0, ymax=8]
\addplot[->, style=dotted, very thick] coordinates{
(6,6)
(8,6)};
\addplot[->, style=dotted, very thick] coordinates{
(6,6)
(6,8)};
\addplot [pattern = north east lines,  draw=white] coordinates{
(6,6)
(8,6)
(8,8)
(6,8)
(6,6)
};
\addplot[only marks] coordinates{
(0,0)
(3,3)
(3,4)
(4,3)
(6,6)
};

\end{axis}
\end{tikzpicture}$\quad$
\begin{tikzpicture}[scale=.75]
 \begin{axis}[grid=both, xmin=0,xmax=8, ymin=0, ymax=8]
\addplot[->, style=dotted, very thick] coordinates{
(3,3)
(8,3)};
\addplot[->, style=dotted, very thick] coordinates{
(3,3)
(3,8)};
\addplot [pattern = north east lines,  draw=white] coordinates{
(3,3)
(8,3)
(8,8)
(3,8)
(3,3)
};
\addplot[only marks] coordinates{
(0,0)
(3,3)
};

\end{axis}
\end{tikzpicture}
\end{center}
Now we do the same with $[(5,3),(3,4)]_{(6,7)}$.
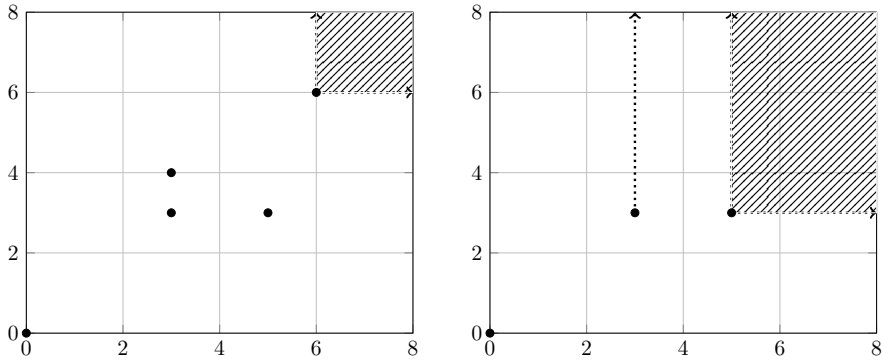
\begin{center}
\begin{tikzpicture}[scale=.75]
 \begin{axis}[grid=both, xmin=0,xmax=8, ymin=0, ymax=8]
 \addplot[->, style=dotted, very thick] coordinates{
(6,6)
(8,6)};
\addplot[->, style=dotted, very thick] coordinates{
(6,6)
(6,8)};
\addplot [pattern = north east lines,  draw=white] coordinates{
(6,6)
(8,6)
(8,8)
(6,8)
(6,6)
};\addplot[only marks] coordinates{
(0,0)
(3,3)
(3,4)
(5,3)
(6,6)
};
\end{axis}
\end{tikzpicture}$\quad$
\begin{tikzpicture}[scale=.75]
 \begin{axis}[grid=both, xmin=0,xmax=8,ymin=0, ymax=8]
\addplot[->, style=dotted, very thick] coordinates{
(5,3)
(8,3)};
\addplot[->, style=dotted, very thick] coordinates{
(3,3)
(3,8)};
\addplot[->, style=dotted, very thick] coordinates{
(5,3)
(5,8)};
\addplot [pattern = north east lines,  draw=white] coordinates{
(5,3)
(8,3)
(8,8)
(5,8)
(5,3)
};
\addplot[only marks] coordinates{
(0,0)
(3,3)
(5,3)
};

\end{axis}
\end{tikzpicture}
\end{center}
And finally with $[(3,4),(4,4)]_{(6,6)}$:
\begin{center}
\begin{tikzpicture}[scale=.75]
 \begin{axis}[grid=both, xmin=0,xmax=8, ymin=0, ymax=8]
\addplot[->, style=dotted, very thick] coordinates{
(6,6)
(8,6)};
\addplot[->, style=dotted, very thick] coordinates{
(6,6)
(6,8)};
\addplot [pattern = north east lines,  draw=white] coordinates{
(6,6)
(8,6)
(8,8)
(6,8)
(6,6)
};
\addplot[only marks] coordinates{
(0,0)
(3,3)
(4,4)
(6,6)
};

\end{axis}
\end{tikzpicture}$\quad$
\begin{tikzpicture}[scale=.75]
 \begin{axis}[grid=both, xmin=0,xmax=8, ymin=0, ymax=8]
\addplot[->, style=dotted, very thick] coordinates{
(6,6)
(8,6)};
\addplot[->, style=dotted, very thick] coordinates{
(6,6)
(6,8)};
\addplot [pattern = north east lines,  draw=white] coordinates{
(6,6)
(8,6)
(8,8)
(6,8)
(6,6)
};
\addplot[only marks] coordinates{
(0,0)
(3,3)
(4,4)
(5,5)
(6,6)
};

\end{axis}
\end{tikzpicture}
\end{center}
\end{example}

It is clear that performing successively the operations $b+c-a$
with $b\ge a, c\ge a$, we obtain a chain of subsemigroups of
$\mathbb{N}^2$ (not necessarily good) contained in the Arf closure
of $S$. We also know this procedure has to end after a finite
number of steps. So we get an Arf semigroup $U$ (not necessarily
good) containing $S$ and contained in its Arf closure $T$. The
following example shows that, unlike in the numerical semigroup
case, these semigroups might not agree.
\begin{example}\label{ex:arf-not-good-closure}
 Let $S$ be the good semigroup with \[
 \mathrm{Small}(S)=\{(0,0), (3,3), (4,4), (5,4), (4,6),(6,6)\}\]
 (and thus conductor equal to $(6,6)$). Then $T\setminus U=\{(4,5)\}$. Next picture depicts $U\cap[0,8]^2$ and the good semigroup $T$.

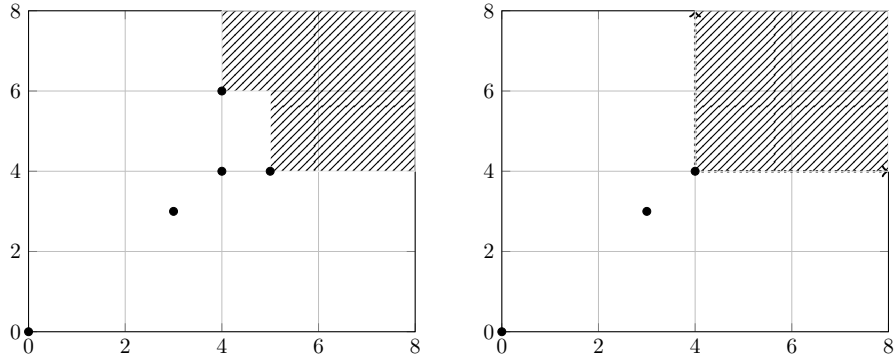
\begin{center}
\begin{tikzpicture}[scale=.75]
 \begin{axis}[grid=both, xmin=0,xmax=8, ymin=0, ymax=8]
\addplot[only marks] coordinates{
(0,0)
(3,3)
(4,4)
(4,6)
%(4,7)
%(4,8)
(5,4)
%(5,5)
%(5,6)
%(5,7)
%(5,8)
%(6,4)
%(6,5)
%(6,6)
%(6,7)
%(6,8)
%(7,4)
%(7,5)
%(7,6)
%(7,7)
%(7,8)
%(8,4)
%(8,5)
%(8,6)
%(8,7)
%(8,8)
};
\addplot [pattern = north east lines,  draw=white] coordinates{
    (4,8)
    (4,6)
    (5,6)
    (5,4)
    (8,4)
    (8,8)
    (4,8)
};
\end{axis}
\end{tikzpicture} $\quad$
\begin{tikzpicture}[scale=.75]
\begin{axis}[grid=both, xmin=0,xmax=8, ymin=0, ymax=8]
\addplot[->, style=dotted, very thick] coordinates{
(4,4)
(8,4)};
\addplot[->, style=dotted, very thick] coordinates{
(4,4)
(4,8)};
\addplot [pattern = north east lines,  draw=white] coordinates{
(4,4)
(8,4)
(8,8)
(4,8)
(4,4)
};
\addplot[only marks] coordinates{
(0,0)
(3,3)
(4,4)
};
\end{axis}
\end{tikzpicture}
\end{center}
\end{example}

\begin{question}
    Notice that in Example \ref{ex:arf-not-good-closure}, if we close $U$ under infima,
 then the resulting monoid completely covers $T$. So the natural
question is that if this procedure always guarantees that the
output will be the (good) Arf closure of the initial good
semigroup. Notice that computationally speaking, it is easier to
perform the operations with the multiplicity trees rather than
doing the saturation under $a+b-c$ and then taking infima.
\end{question}
%
%\section{ToDO}
%{\color{blue}
%\begin{enumerate}[(A)]
%
%\item Find a good (minimal) system of generators for the Arf closure
%
%%\item Change membership for semigroups given by generators (code).
%
%%\item Arf closure.
%
%%\item Relative ideals and the canonical relative ideal? In particular, can we always represent them with finite data? Do they have a unique minimal generating system with finitely many elements?
%
%%\item Test for good relative ideals? (Seems easy if ideals are finitely generated as good semigroups)
%
%\item Almost symmetry.
%
%%\item Type for some families?
%
%\end{enumerate}
%}


\begin{thebibliography}{10}

\bibitem{a-u} V. Barucci, M. D'Anna, R. Fr\"oberg, Analytically unramified one-dimensional semilocal rings and their value semigroups, J. Pure Appl. Alg. \textbf{147} (2000), 215-254.

\bibitem{two} V. Barucci, M. D'Anna, R. Fr\"oberg, The semigroup of values of a one-dimensional local ring with two minimal primes, Comm. Algebra \textbf{28}(8) (2000), 3 607-3633.

\bibitem{arf} V. Barucci, M. D'Anna, R. Fr\"oberg, Arf characters of an algebroid curve, JP Journal of Algebra, Number Theory and Applications, vol. 3 \textbf{2} (2003), 219-243.

\bibitem{apery-alg} V. Barucci, M. D'Anna, R. Fr\"oberg, The Apery algorithm for a plane singularity with two branches, Beitr\"age zur Algebra und Geometrie \textbf{46} (2005), 1-18.

\bibitem{bdf} V. Barucci, D. E. Dobbs, M. Fontana, Maximality Properties in Numerical Semigroups and Applications to One-Dimensional Analytically Irreducible Local Domains, Memoirs of the Amer.  Math. Soc. \textbf{598} (1997).

\bibitem{c-d-gz} A. Campillo,  F. Delgado, S. M. Gusein-Zade,  On generators of the semigroup of a plane curve singularity. J. London Math. Soc. (2) \textbf{60} (1999), 420-430.

\bibitem{C-D-K} A. Campillo, F. Delgado, K. Kiyek, Gorenstein properties and symmetry for one-dimensional local Cohen-Macaulay rings, Manuscripta Math.  \textbf{83} (1994), 405-423.


\bibitem{danna1} M. D'Anna, The canonical module of a one-dimensional reduced local ring, Comm. Algebra \textbf{25} (1997), 2939–-2965.

\bibitem{danna} M. D'Anna, Ring and semigroup constructions, in Multiplicative
Ideal Theory and Factorization Theory - Commutative and Non-Commutative
Perspectives, Springer proccedings in mathematics and statistics, vol. 170 (2016), 97-115, Springer.

\bibitem{felix} F. Delgado, The semigroup of values of a curve singularity with several branches, Manuscripta Math. \textbf{59} (1987), 347-374.

\bibitem{delgado} F. Delgado, Gorenstein curves and symmetry of the semigroup of value, in Manuscripta Math.  \textbf{61} (1988), 285-296.
Perspectives, Springer.

\bibitem{numericalsgps} M. Delgado, P. A. Garc\'{\i}a-S\'{a}nchez, J. Morais,
\lq\lq NumericalSgps\rq\rq, A GAP package for numerical semigroups, Version 1.0.1  (2015), (Refereed GAP package), \url{http://www.fc.up.pt/cmup/mdelgado/numericalsgps}..

\bibitem{gap}
  The GAP~Group, GAP -- Groups, Algorithms, and Programming,
  Version 4.7.5,
  2014,
  \url{http://www.gap-system.org}.

\bibitem{garcia} A. Garc\'ia, Semigroups associated to singular points of plane curves, J. Reine Angew. Math.  \textbf{336} (1982), 165-184.

\bibitem{K-T-S} P. Korell, M. Schulze, L. Tozzo, Duality of value semigroups, J. Comm. Alg. \textbf{to appear}, arXiv:1510.04072.


\bibitem{fg} J. C. Rosales y P. A. Garc\'ia-S\'anchez, Finitely generated commutative monoids, Nova Science Publishers, Inc., New York, 1999.

\bibitem{arf-num} J. C. Rosales, P. A. Garc\'ia-S\'anchez, J. I. Garc\'ia-Garc\'ia, and M. B. Branco, Arf numerical semigroups, J. Algebra \textbf{276} (2004), 3-12.

\bibitem{salem-th} Sahar Saleh. Calcul de la fonction d'Artin d'une singularit \'e plane. Math\'ematiques [math]. Universit \'e d'Angers, 2010.

\bibitem{waldi} R. Waldi, Wertehalbgruppe und Singularit\"at einer ebenen algebraischen Kurve, Dissertation, Regensburg, 1972.

\bibitem{zito} G. Zito, Arf semigroups with prescribed multiplicity tree,
preprint (2017).
\end{thebibliography}
\end{document}